\documentclass{tac}

\usepackage{amsfonts, amssymb, amsmath, stmaryrd, etoolbox}
\usepackage{comment}
\usepackage{mathtools}
\usepackage{graphicx,caption,subcaption}
\usepackage{todonotes}
\usepackage{xcolor}
\usepackage[T1]{fontenc}
\apptocmd{\sloppy}{\hbadness 10000\relax}{}{}

\usepackage[inline]{enumitem}
\setlist{itemsep=0em, topsep=0em, parsep=0em}
\setlist[enumerate]{label=(\alph*)}

\usepackage{tikz}
\usepackage[all,2cell]{xy}
\usepackage[all,2cell]{xy}
\usetikzlibrary{
	matrix,arrows,shapes,decorations.markings,decorations.pathreplacing
}
\tikzset{
	rewritenode/.style={shape=circle,draw,minimum size=1em}
}
\tikzset{
	RWopen/.style={shape=circle,draw=black,fill=white,scale=0.5,font=\Huge}
}
\tikzset{
	RWclosed/.style={shape=circle,fill=black,scale=0.5,font=\Huge}
}
\tikzset{CDnode/.style={shape=circle,fill=white,scale=.5}}
\tikzset{->-/.style={decoration={markings,mark=at position .5 with {\arrow{>}}},postaction={decorate}}}

\usepackage{hyperref}
\definecolor{hyperrefcolor}{rgb}{0,0,0.7}
\hypersetup{
	colorlinks,
	linkcolor={hyperrefcolor},
	citecolor={hyperrefcolor},
	urlcolor={hyperrefcolor}
}

\renewcommand{\epsilon}{\varepsilon}

\newcommand{\cat}[1]{\mathbf{#1}}
\newcommand{\dblcat}[1]{\mathbb{#1}}
\renewcommand{\t}[1]{\textup{#1}}

\newcommand{\from}{\colon}
\newcommand{\xto}[1]{\xrightarrow{#1}}

\newcommand{\tospan}{\xrightarrow{\mathrm{sp}}}
\newcommand{\tocospan}{\xrightarrow{\mathrm{csp}}}

\newcommand{\bispmap}[1]{\mathbf{Sp(#1)}}

\newcommand{\bimonspcsp}[1]{\mathbf{MonicSp(Csp(#1))}}
\newcommand{\dblmonspcsp}[1]{\mathbb{M}\mathbf{onicSp(Csp(#1))}}

\makeatletter
\def\slashedarrowfill@#1#2#3#4#5{%
	$\m@th\thickmuskip0mu\medmuskip\thickmuskip\thinmuskip\thickmuskip
	\relax#5#1\mkern-7mu%
	\cleaders\hbox{$#5\mkern-2mu#2\mkern-2mu$}\hfill
	\mathclap{#3}\mathclap{#2}%
	\cleaders\hbox{$#5\mkern-2mu#2\mkern-2mu$}\hfill
	\mkern-7mu#4$%
}
\def\rightslashedarrowfill@{%
	\slashedarrowfill@\relbar\relbar\mapstochar\rightarrow}
\newcommand{\xslashedrightarrow}[2][]{%
	\ext@arrow 0055{\rightslashedarrowfill@}{#1}{#2}}
\makeatother

\newcommand{\hto}{\xslashedrightarrow{}}

\DeclareMathOperator{\id}{id}

\newtheorem{thm}{Theorem}[section]
\newtheorem{lem}[thm]{Lemma}

\theoremstyle{remark}

\theoremstyle{definition}
 
\newtheorem{defn}[thm]{Definition}


\copyrightyear{2017}


\begin{document}
\sloppy	

\title{Spans of cospans in a topos}
\author{Daniel Cicala and Kenny Courser}
\maketitle

\address{Department of Mathematics, University of California\\
 Riverside, CA 92521, USA\\[5pt]}

\eaddress{%
cicala@math.ucr.edu \\
\null \hspace{2.6em} courser@math.ucr.edu}

\begin{abstract}
For a topos $\cat{T}$, 
there is a bicategory
$\cat{MonicSp(Csp(T))}$ 
whose objects are those of $\mathbf{T}$,
morphisms are cospans in $\mathbf{T}$, 
and 2-morphisms are 
isomorphism classes of monic spans of cospans 
in $\mathbf{T}$. 
Using a result of Shulman, 
we prove that $\cat{MonicSp(Csp(T))}$ is 
symmetric monoidal, and moreover, that it is compact closed in the sense of Stay.
We provide an application which illustrates how
to encode double pushout rewrite rules
as $2$-morphisms inside a compact closed sub-bicategory of
$\mathbf{MonicSp(Csp(Graph))}$.
\end{abstract}

\section{Introduction} 
\label{sec:Introduction}

There has been extensive work done on bicategories involving spans or cospans in some way.
Given a finitely complete category $\cat{D}$,
B\'{e}nabou \cite{Be} 
was the first to construct the bicategory
$\bispmap{D}$ consisting of
objects of $\cat{D}$, spans in $\cat{D}$,
and maps of spans in $\cat{D}$, and this was in fact one of the 
first bicategories ever constructed.
Later, Stay showed that $\bispmap{D}$ 
is a compact closed symmetric monoidal bicategory
\cite{Stay}.
Various other authors have considered bicategories and higher categories with maps of spans or spans of spans as 2-morphisms \cite{Haug,Hoff,Nie,Reb,Stay}, as well as categories and bicategories in which the morphisms are cospans `decorated' with extra structure \cite{BaezCoyaRebro, BaezFong, BaezFongPollard, BaezPollard,Cour,Fong}.
Here, however, we pursue a different line of thought and study spans of cospans. 

A span of cospans is a 
commuting diagram with shape
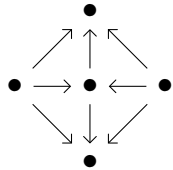
\[
\begin{tikzpicture}
\node (A) at (0,0) {$\bullet$};
\node (B) at (1,1) {$\bullet$};
\node (B') at (1,0) {$\bullet$};
\node (B'') at (1,-1) {$\bullet$};
\node (C) at (2,0) {$\bullet$};
\path[->,font=\scriptsize,>=angle 90]
(A) edge node[above]{$ $} (B)
(A) edge node[above]{$ $} (B')
(A) edge node[above]{$ $} (B'')
(C) edge node[above]{$ $} (B)
(C) edge node[above]{$ $} (B')
(C) edge node[left]{$ $} (B'')
(B') edge[->] node[right]{$ $} (B)
(B') edge[->] node[left]{$ $} (B'');
\end{tikzpicture}
\]
These were found to
satisfy a lax interchange law
by Grandis and Par\'{e}
in their paper on `intercategories' 
\cite{GranPare_Intercats}.
Later on, the first listed author of the present work constructed
a bicategory with isomorphism classes of spans of cospans
as $2$-morphisms \cite{Cic}. 
However, for the interchange law to be invertible, the $2$-morphisms
were restricted to 
spans of cospans inside
a topos $\mathbf{T}$ such that the 
span legs were monomorphisms.

We denote the bicategory of monic spans of cospans
inside of a topos $\cat{T}$ by $\bimonspcsp{T}$. 
It has the objects of $\cat{T}$
for objects,
cospans in $\cat{T}$ for morphisms,
and isomorphism classes of monic spans of cospans
in $\cat{T}$ for $2$-morphisms.
Horizontal composition is given 
by pushouts and
vertical composition is
given by pullbacks.
Thus we require that certain 
limits and colimits exist and, moreover, 
that pushouts preserve monomorphisms. 
This occurs in every topos.


Our first result is that 
the bicategory $\bimonspcsp{T}$ 
is symmetric monoidal. 
The definition of 
symmetric monoidal bicategory is long \cite{Stay}, 
so checking every condition 
by hand is time consuming. 
Shulman 
\cite{Shul} 
provides a less tedious process. 
The idea is to construct an 
`isofibrant pseudo double category' 
that restricts, in a suitable sense, 
to the bicategory 
that we are interested in.  
If this double category is symmetric monoidal, 
then the `restricted' bicategory is
symmetric monoidal as well.  
The advantage of this method 
is its relative efficiency; it is much easier to check that 
a double category is symmetric monoidal
than it is to check that a bicategory is symmetric monoidal.

Our second result is that $\bimonspcsp{T}$ 
is compact closed.
Stay \cite{Stay} has defined a 
compact closed bicategory 
to be a symmetric monoidal bicategory 
in which every object has a dual. 
The details seem more involved than in the ordinary category case due to the coherence laws.  
However, Pstr\k{a}gowski \cite{Piotr} has found a way to avoid checking the worst of these laws: the `swallowtail equations'.

The primary motivation for constructing $\bimonspcsp{T}$ 
is the case when $\cat{T}$ 
is the topos $\cat{Graph}$ 
of directed graphs.
We are also interested in 
certain labeled graphs obtained
from various slice categories of 
$\cat{Graph}$.
This is a useful framework to study 
open graphs; that is, graphs with
a subset of nodes serving as `inputs' and `outputs'
and `rewrite rules' of such \cite{Cic}.
Graphical calculi fit nicely
into this picture as well \cite{Cic_zx}.

Let us illustrate how this 
works with open graphs.
We begin with a compact closed sub-bicategory
of $\bimonspcsp{Graph}$ which we shall call 
	$\cat{Rewrite}$. 
The sub-bicategory $\cat{Rewrite}$ is 1-full and 2-full on edgeless graphs and
the 1-morphisms model open graphs 
by specifying the inputs and outputs
with the legs of the cospans. 
For example, consider the following cospan of graphs.
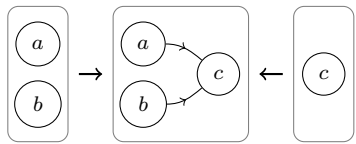
\[
\begin{tikzpicture}
\begin{scope}[shift={(0.7,2)}]
\draw[rounded corners,gray] (-0.3,1.6) rectangle (0.5,3.4);
\node[circle,draw] (ai) at (0.1,2.9) {\scriptsize $a$};
\node[circle,draw] (bi) at (0.1,2.1) {\scriptsize $b$};
\end{scope}
\begin{scope}
\draw[rounded corners,gray] (1.8,3.6) rectangle (3.6,5.4);
\node[circle,draw] (a1) at (2.2,4.9) {\scriptsize $a$};
\node[circle,draw] (b1) at (2.2,4.1) {\scriptsize $b$};
\node[circle,draw] (c1) at (3.2,4.5) {\scriptsize $c$};
\draw [->-] (a1) to [out=0,in=140] (c1);
\draw [->-] (b1) to [out=0,in=220] (c1);
\end{scope}
\begin{scope}[shift={(-0.6,2)}]
\draw[rounded corners,gray] (4.8,1.6) rectangle (5.6,3.4);
\node[circle,draw] (co) at (5.2,2.5) {\scriptsize $c$};
\end{scope}
\node (v1) at (1.2,4.5) {};
\node (v2) at (1.8,4.5) {};
\node (v3) at (4.2,4.5) {};
\node (v4) at (3.6,4.5) {};
\draw [->,thick] (v1) edge (v2);
\draw [->,thick] (v3) edge (v4);
\end{tikzpicture}
\]
The node labels indicate the graph morphism behaviors. 
Here, the nodes $a$ and $b$ are inputs 
and $c$ is an output. 
The use of the terms \emph{inputs} and \emph{outputs} 
is justified by the composition,
which can be thought of as 
connecting the inputs of one graph
to compatible outputs of another graph. 
This is made precise with pushouts.  
For instance, we can compose
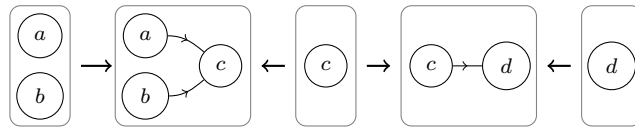
\[
\begin{tikzpicture}
\begin{scope}[shift={(0.1,2)}]
\draw[rounded corners,gray] (-0.3,1.7) rectangle (0.5,3.3);
\node[circle,draw] (ai) at (0.1,2.9) {\scriptsize $a$};
\node[circle,draw] (bi) at (0.1,2.1) {\scriptsize $b$};
\end{scope}
\begin{scope}[shift={(-0.6,0)}]
\draw[rounded corners,gray] (1.8,3.7) rectangle (3.6,5.3);
\node[circle,draw] (a1) at (2.2,4.9) {\scriptsize $a$};
\node[circle,draw] (b1) at (2.2,4.1) {\scriptsize $b$};
\node[circle,draw] (c1) at (3.2,4.5) {\scriptsize $c$};
\draw [->-] (a1) to [out=0,in=140] (c1);
\draw [->-] (b1) to [out=0,in=220] (c1);
\end{scope}
\begin{scope}[shift={(-0.6,1.9)}]
\draw[rounded corners,gray] (4.2,1.8) rectangle (5,3.4);
\node[circle,draw] (cm) at (4.6,2.6) {\scriptsize $c$};
\end{scope}
\begin{scope}[shift={(-1.6,0.1)}]
\draw[rounded corners,gray] (6.6,3.6) rectangle (8.4,5.2);
\node[circle,draw] (c2) at (7,4.4) {\scriptsize $c$};
\node[circle,draw] (d2) at (8,4.4) {\scriptsize $d$};
\draw [->-] (c2) to (d2);
\end{scope}
\begin{scope}[shift={(-2.1,2)}]
\draw[rounded corners,gray] (9.5,1.7) rectangle (10.3,3.3);
\node[circle,draw] (do) at (9.9,2.5) {\scriptsize $d$};
\end{scope}
\node (v1) at (0.6,4.5) {};
\node (v2) at (1.3,4.5) {};
\node (v4) at (3,4.5) {};
\node (v3) at (3.6,4.5) {};
\node (v5) at (4.4,4.5) {};
\node (v6) at (5,4.5) {};
\node (v8) at (6.8,4.5) {};
\node (v7) at (7.4,4.5) {};
\draw [->,thick]  (v1) edge (v2);
\draw [->,thick] (v3) edge (v4);
\draw [->,thick] (v5) edge (v6);
\draw [->,thick] (v7) edge (v8);
\end{tikzpicture}
\]
to obtain
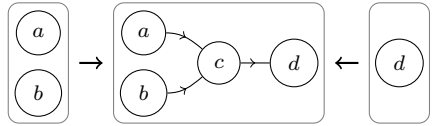
\[
\begin{tikzpicture}
\begin{scope}[shift={(0.7,0)}]
\draw[rounded corners,gray] (-0.3,1.7) rectangle (0.5,3.3);
\node[circle,draw] (ai) at (0.1,2.9) {\scriptsize $a$};
\node[circle,draw] (bi) at (0.1,2.1) {\scriptsize $b$};
\end{scope}
\begin{scope}[shift={(0,-2)}]
\draw[rounded corners,gray] (1.8,3.7) rectangle (4.6,5.3);
\node[circle,draw] (a1) at (2.2,4.9) {\scriptsize $a$};
\node[circle,draw] (b1) at (2.2,4.1) {\scriptsize $b$};
\node[circle,draw] (c1) at (3.2,4.5) {\scriptsize $c$};
\node[circle,draw] (d1) at (4.2,4.5) {\scriptsize $d$};
\draw [->-] (a1) to [out=0,in=140] (c1);
\draw [->-] (b1) to [out=0,in=220] (c1);
\draw [->-] (c1) to (d1);
\end{scope}
\begin{scope}[shift={(-0.7,0)}]
\draw[rounded corners,gray] (5.9,1.7) rectangle (6.7,3.3);
\node[circle,draw] (co) at (6.3,2.5) {\scriptsize $d$};
\end{scope}
\node (v1) at (1.8,2.5) {};
\node (v2) at (1.2,2.5) {};
\node (v3) at (4.6,2.5) {};
\node (v4) at (5.2,2.5) {};
\draw [<-,thick]  (v1) edge (v2);
\draw [<-,thick] (v3) edge (v4);
\end{tikzpicture}
\]
A 2-morphism in $\cat{Rewrite}$ 
is the rewriting of one graph 
into another in a way that 
preserves inputs and outputs. 
For instance,
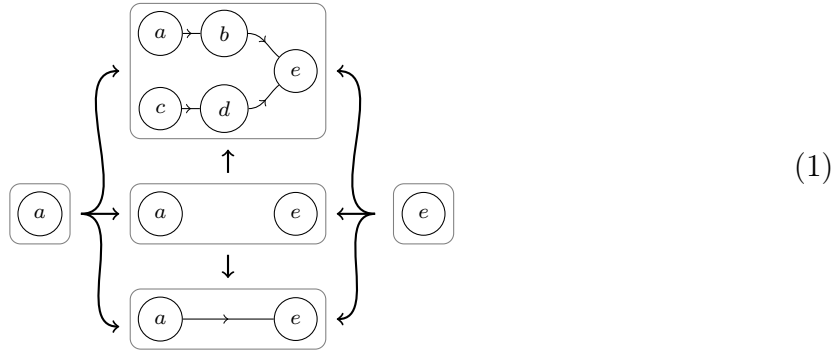
\begin{equation}
\label{diag:MonoidUnit}
\begin{tikzpicture}[baseline=(current bounding box.center)]
	\begin{scope}[shift={(0.3,0.7)}]
	\draw[rounded corners,gray] (-0.5,2.1) rectangle (0.3,2.9);
	\node[circle,draw] (ai) at (-0.1,2.5) {\scriptsize $a$};
	\end{scope}
	\begin{scope}[shift={(-0.1,0.6)}]
	\draw[rounded corners,gray] (1.5,3.6) rectangle (4.1,5.4);
	\node[circle,draw] (a1) at (1.9,5) {\scriptsize $a$};
	\node[circle,draw] (b1) at (2.75,5) {\scriptsize $b$}; 
	\node[circle,draw] (c1) at (1.9,4) {\scriptsize $c$}; 
	\node[circle,draw] (d1) at (2.75,4) {\scriptsize $d$};
	\node[circle,draw] (e1) at (3.7,4.5) {\scriptsize $e$};
	\draw [->-] (a1) to (b1);
	\draw [->-] (b1) to [out=0,in=140] (e1);
	\draw [->-] (c1) to (d1);
	\draw [->-] (d1) to [out=0,in=220] (e1);
	\end{scope}
	\begin{scope}[shift={(0,0.7)}]
	\draw[rounded corners,gray] (1.4,2.1) rectangle (4,2.9);
	\node[circle,draw] (a2) at (1.8,2.5) {\scriptsize $a$};
	\node[circle,draw] (b2) at (3.6,2.5) {\scriptsize $e$};
	\end{scope}
	\begin{scope}[shift={(-0.2,1.3)}]
	\draw[rounded corners,gray] (1.6,0.1) rectangle (4.2,0.9);
	\node[circle,draw] (a3) at (2,0.5) {\scriptsize $a$};
	\node[circle,draw] (b3) at (3.8,0.5) {\scriptsize $e$};
	\draw [->-] (a3) to (b3);
	\end{scope}
	\begin{scope}[shift={(0.2,0.3)}]
	\draw[rounded corners,gray] (4.7,2.5) rectangle (5.5,3.3);
	\node[circle,draw] (bo) at (5.1,2.9) {\scriptsize $e$};
	\end{scope}
	\node (v1) at (0.6,3.2) {};
	\node (v2) at (1.4,5.1) {};
	\node (v3) at (1.4,3.2) {};
	\node (v4) at (1.4,1.7) {};
	\node (v5) at (4.8,3.2) {};
	\node (v6) at (4,5.1) {};
	\node (v7) at (4,3.2) {};
	\node (v8) at (4,1.8) {};
	\node (v9) at (2.7,3.6) {};
	\node (v10) at (2.7,4.2) {};
	\node (v11) at (2.7,2.8) {};
	\node (v12) at (2.7,2.2) {};
	\draw [->,thick]  (v1) edge [out=0,in=180] (v2);
	\draw [->,thick] (v1) edge (v3);
	\draw [->,thick] (v1) edge [out=0,in=180] (v4);
	\draw [->,thick] (v5) edge [out=180,in=0] (v6);
	\draw [->,thick] (v5) edge (v7);
	\draw [->,thick] (v5) edge [out=180,in=0] (v8);
	\draw [->,thick] (v9) edge (v10);
	\draw [->,thick] (v11) edge (v12);
\end{tikzpicture}
\end{equation}
The bicategory $\cat{Rewrite}$ is not of interest for its own sake.
It serves as an ambient context in which to freely generate 
a compact closed bicategory 
from some collection of morphisms and 2-morphisms. 
Of course, the choice of collection depends on one's interests.  
 
The structure of the paper is as follows.  
In Section 
	\ref{sec:Span cospan bicats}, 
we introduce the bicategory $\bimonspcsp{T}$.  
In Section 
	\ref{sec:DoubleCategories}, 
we review how symmetric monoidal double categories can be used 
to show that certain bicategories have a symmetric monoidal structure,
and moreover, when a symmetric monoidal bicategory is compact closed.  
Nothing in this section is new, 
but we use these results 
in Section \ref{sec:SpansCospans} 
to show that $\bimonspcsp{T}$ is compact closed.
Finally, in Section 
	\ref{sec:Applications}, 
we discuss an application that illustrates the consequeces of the bicategory $\mathbf{MonicSp(Csp(Graph))}$ being compact closed
in regard to rewriting open graphs.

\section{The bicategory \texorpdfstring{$\mathbf{MonicSp(Csp(T))}$}{}} 
\label{sec:Span cospan bicats}

In this section, we recall the bicategory $\bimonspcsp{T}$ and some important related concepts.  
Throughout this paper, 
$\cat{T}$ is a topos.

Spans of cospans were considered by Kissinger
in his thesis \cite{Kiss} in the context of rewriting,
and also by
Grandis and Par\'{e} 
	\cite{GranPare_Intercats} 
who found a lax interchange law. 
Later, the first listed author of the present work showed that the interchange law is invertible
when restricting attention to a topos $\cat{T}$ and monic spans \cite{Cic}, meaning that each morphism is a monomorphism. 
This gives a bicategory $\bimonspcsp{T}$ with 
$\cat{T}$-objects for objects, 
cospans for morphisms, 
and isomorphism classes of monic spans of cospans for 2-morphisms. 
A monic span of cospans is a commuting diagram of the form
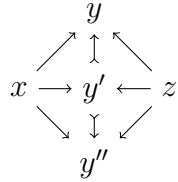
\[
\begin{tikzpicture}[]
	\node (A) at (0,0) {$x$};
	\node (B) at (1,1) {$y$};
	\node (B') at (1,0) {$y'$};
	\node (B'') at (1,-1) {$y''$};
	\node (C) at (2,0) {$z$};
	\path[->,font=\scriptsize]
	(A) edge node[above]{$ $} (B)
	(A) edge node[above]{$ $} (B')
	(A) edge node[above]{$ $} (B'')
	(C) edge node[above]{$ $} (B)
	(C) edge node[above]{$ $} (B')
	(C) edge node[left]{$ $} (B'')
	(B') edge[>->] node[right]{$ $} (B)
	(B') edge[>->] node[left]{$ $} (B'');
\end{tikzpicture}
\]
where the '$\rightarrowtail$' arrows denotes a monomorphism, and two of these are isomorphic if there is an isomorphism $\theta$ such that the following diagram commutes:
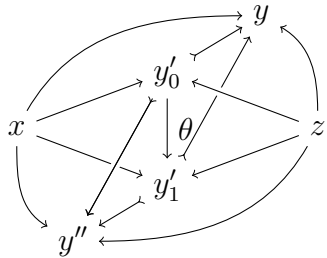
\[
\begin{tikzpicture}[scale=1]
	\node (X) at (-2,0) {$x$};
	\node (L) at (1.25,1.5) {$y$};
	\node (Y) at (2,0) {$z$};
	\node (S) at (-1.25,-1.5) {$y''$};
	\node (S'1) at (0,0.75) {$y'_0$};
	\node (S'2) at (0,-0.75) {$y'_1$};
	\draw [->] (X) edge[in=180,out=60] (L);
	\draw [->] (X) edge[] (S'1);
	\draw [->] (X) edge[] (S'2);
	\draw [->] (X) edge[out=-90,in=150] (S);
	\draw [->] (Y) edge[out=90,in=-30] (L);
	\draw [->] (Y) edge[] (S'2);
	\draw [->] (Y) edge[out=-120,in=0] (S);
	\draw [>->] (S'1) edge[] (L);
	\draw [>->] (S'1) edge[white,line width=3.5pt] (S);
	\draw [>->] (S'1) edge[] (S);
	\draw [->] (S'1) edge[] (S);
	\draw [->] (S'1) edge node[right]{$\theta$} (S'2);
	\draw [>->] (S'2) edge[] (L);
	\draw [>->] (S'2) edge[] (S);
	\draw [->] (Y) edge[white,line width=3.5pt] (S'1);
	\draw [->] (Y) edge[] (S'1);
\end{tikzpicture}
\]
As usual, composition of morphisms is by pushout. Given vertically composable 2-morphisms
	
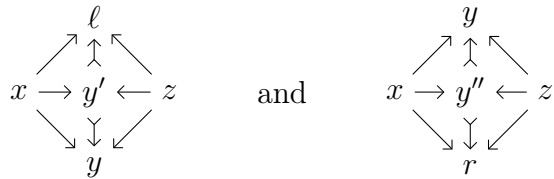
\[
	\begin{tikzpicture}
		\node (B) at (2,1) {$\ell$};
		\node (A') at (1,0) {$x$};
		\node (B') at (2,0) {$y'$};
		\node (C') at (3,0) {$z$};
		\node (E) at (2,-1) {$y$};
		\node (B'') at (7,1) {$y$};
		\node (A''') at (6,0) {$x$};
		\node (B''') at (7,0) {$y''$};
		\node (C''') at (8,0) {$z$};
		\node (E'') at (7,-1) {$r$};
		\node at (4.5,0) { and };
		\path[->,font=\scriptsize,>=angle 90]
               (C''') edge node {$$} (E'')
               (A''') edge node {$$} (E'')
               (C''') edge node {$$} (B'')
			(A''') edge node {$$} (B'')
               (C') edge node {$$} (E)
               (A') edge node {$$} (E)
               (C') edge node [above]{$$}(B)
               (A') edge node[above]{$$}(B)
               (A') edge node[above]{$$}(B')
			(C')edge node[above]{$$}(B')
			(B') edge[>->] node[left]{$$} (B)
			(B') edge[>->] node[right]{$$} (E)
			(A''')edge node[above]{$$}(B''')
			(C''')edge node[above]{$$}(B''')
			(B''') edge[>->] node[left]{$$} (B'')
			(B''') edge[>->] node[right]{$$} (E'');
	\end{tikzpicture}
	\]
their vertical composite is given by
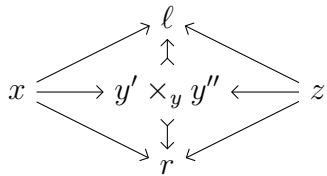
\[
	\begin{tikzpicture}
		\node (B) at (2,1) {$\ell$};
		\node (A') at (0,0) {$x$};
		\node (B') at (2,0) {$y' \times_{y} y''$};
		\node (C') at (4,0) {$z$};
		\node (E) at (2,-1) {$r$};
		\path[->,font=\scriptsize,>=angle 90]
                     (C') edge node {$$} (E)
                     (A') edge node {$$} (E)
                     (C') edge node [above]{$$}(B)
                     (A') edge node[above]{$$}(B)
                     (A') edge node[above]{$$}(B')
		(C')edge node[above]{$$}(B')
		(B') edge[>->] node[left]{$$} (B)
		(B') edge[>->] node[right]{$$} (E);
	\end{tikzpicture}
	\]
The legs of the inner span are monic because pullbacks preserve monomorphisms. Given two horizontally composable 2-morphisms
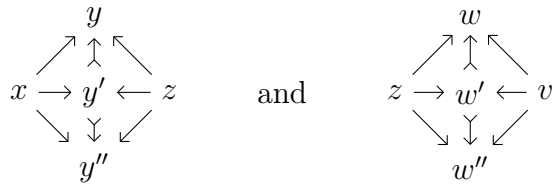
\[
	\begin{tikzpicture}
		\node (B) at (2,1) {$y$};
		\node (A') at (1,0) {$x$};
		\node (B') at (2,0) {$y^\prime$};
		\node (C') at (3,0) {$z$};
		\node (E) at (2,-1) {$y''$};

		\node (B'') at (7,1) {$w$};

		\node (A''') at (6,0) {$z$};
		\node (B''') at (7,0) {$w'$};
		\node (C''') at (8,0) {$v$};

		\node (E'') at (7,-1) {$w''$};
		
		\node at (4.5,0) { and };
		\path[->,font=\scriptsize,>=angle 90]
                     (C''') edge node {$$} (E'')
                     (A''') edge node {$$} (E'')
                     (C''') edge node {$$} (B'')
		(A''') edge node {$$} (B'')
                     (C') edge node {$$} (E)
                     (A') edge node {$$} (E)
                     (C') edge node [above]{$$}(B)
                     (A') edge node[above]{$$}(B)
                     (A') edge node[above]{$$}(B')
		(C')edge node[above]{$$}(B')
		(B') edge[>->] node[left]{$$} (B)
		
		(B') edge[>->] node[right]{$$} (E)

		(A''')edge node[above]{$$}(B''')
		(C''')edge node[above]{$$}(B''')

		(B''') edge[>->] node[left]{$$} (B'')

		(B''') edge[>->] node[right]{$$} (E'');
		
	\end{tikzpicture}
	\]
their horizontal composite is given by
\[
	\begin{tikzpicture}
		\node (B) at (2,1) {$y+_z w$};
		\node (A') at (0,0) {$x$};
		\node (B') at (2,0) {$y^\prime +_z w^\prime$};
		\node (C') at (4,0) {$v$};
		\node (E) at (2,-1) {$y'' +_z w''$};

		\path[->,font=\scriptsize,>=angle 90]
                   
                     (C') edge node {$$} (E)
                     (A') edge node {$$} (E)
                     (C') edge node [above]{$$}(B)
                     (A') edge node[above]{$$}(B)
                     (A') edge node[above]{$$}(B')
		(C')edge node[above]{$$}(B')
		(B') edge[>->] node[left]{$$} (B)
		
		(B') edge[>->] node[right]{$$} (E);

	\end{tikzpicture}
	\]
where the legs of the inner span are monic by a previous result of the first author \cite[Lem.~2.2]{Cic}. The monoidal structure is given by coproducts, and this again preserves the legs of the inner spans being monic.
	\[
	\begin{tikzpicture}
		%
		%
		\node (B) at (2,1) {$y$};
		\node (A') at (1,0) {$x$};
		\node (B') at (2,0) {$y^\prime$};
		\node (C') at (3,0) {$z$};
		\node (E) at (2,-1) {$y''$};
		%
		%
		\node (B'') at (5,1) {$w$};
		\node (A''') at (4,0) {$v$};
		\node (B''') at (5,0) {$w'$};
		\node (C''') at (6,0) {$u$};
		\node (E'') at (5,-1) {$w''$};
		%
		%
		\node at (3.5,0) {$+$};
        \node at (6.5,0) {$=$};
        %
        %
        \node (A) at (7.5,0) {$x+v$};
        \node (C) at (9.5,1) {$y+w$};
        \node (D) at (9.5,0) {$y^\prime + w^\prime$};
        \node (F) at (9.5,-1) {$y'' + w''$};
        \node (G) at (11.5,0) {$z+u$};
		\path[->,font=\scriptsize,>=angle 90]
                     (A) edge node {$$} (C)
                     (A) edge node {$$} (D)
                     (A) edge node {$$} (F)
                     (D) edge [>->] node {$$} (C)
                     (D) edge [>->] node {$$} (F)
                     (G) edge node {$$} (C)
                     (G) edge node {$$} (D)
                     (G) edge node {$$} (F)
                     (C''') edge node {$$} (E'')
                     (A''') edge node {$$} (E'')
                     (C''') edge node {$$} (B'')
					(A''') edge node {$$} (B'')
                     (C') edge node {$$} (E)
                     (A') edge node {$$} (E)
                     (C') edge node [above]{$$}(B)
                     (A') edge node[above]{$$}(B)
                     (A') edge node[above]{$$}(B')
					(C')edge node[above]{$$}(B')
					(B') edge[>->] node[left]{$$} (B)		
					(B') edge[>->] node[right]{$$} (E)		
					(A''')edge node[above]{$$}(B''')
					(C''')edge node[above]{$$}(B''')		
					(B''') edge[>->] node[left]{$$} (B'')		
					(B''') edge[>->] node[right]{$$} (E'');		
	\end{tikzpicture}
	\]
This bicategory is explored further in Section 
	\ref{subsec:Replacing a node in a network} 
where we take $\cat{T}$ to be the topos of directed graphs.

\section{Double categories and duality} 
\label{sec:DoubleCategories}

Bicategories are nice, 
but symmetric monoidal bicategories are much nicer.
Unfortunately, checking the coherence conditions 
for a symmetric monoidal bicategory 
is daunting due to the sheer number of them.  
However, using a result of Shulman 
	\cite{Shul}, 
we can circumvent checking these conditions by
promoting our bicategories to double categories 
and showing that the double categories are symmetric monoidal.  
This involves proving that a pair of categories are symmetric monoidal, 
which is a much more manageable task.  
Even better than symmetric monoidal bicategories 
are those that are compact closed. 
The notion of compact closedness
we consider is given by Stay \cite{Stay}.
After summarizing Shulman's and Stay's work,
we use their machinery
to show that $\bimonspcsp{T}$
is not only symmetric monoidal, but also compact closed.

\subsection{Monoidal bicategories from monoidal double categories}
\label{subsec:DoubleCategories}

Double categories, 
or pseudo double categories to be precise, 
have been studied by Fiore \cite{Fiore} 
and Par\'{e} and Grandis
	\cite{Gran} among others. 
Before giving a formal definition, 
it is helpful to have the following picture in mind. 
A double category has 2-morphisms that look like this:

\begin{equation}
\label{diag:DblCatSquare}
\raisebox{-0.5\height}{
	\begin{tikzpicture}
	\node (A) at (0,1) {$A$};
	\node (B) at (1,1) {$B$};
	\node (C) at (0,0) {$C$};
	\node (D) at (1,0) {$D$};
	\path[->,font=\scriptsize,>=angle 90]
	(A) edge node[above]{$M$} (B)
	(A) edge node[left]{$f$} (C)
	(B) edge node[right]{$g$} (D)
	(C) edge node[below]{$N$} (D);
	\draw (0.5,.925) -- (0.5,1.075);
	\draw (0.5,-.075) -- (0.5,.075);
	\node () at (0.5,0.5) {\scriptsize{$\Downarrow a$}};
	\end{tikzpicture}
}
\end{equation}
We call $A$, $B$, $C$ and $D$ \textbf{objects} or \textbf{0-cells}, 
$f$ and $g$ \textbf{vertical 1-morphisms}, 
$M$ and $N$ \textbf{horizontal 1-morphisms}, 
and $a$ a \textbf{2-morphism}. 
Note that vertical 1-morphisms go between objects
and 2-morphisms go between horizontal 1-morphisms. 
In our definitions, we denote a double category
with a bold font `$\dblcat{D}$' 
either as a stand-alone letter 
or as the first letter in a longer name.

%
\begin{defn}
	\label{def:DoubleCategory}
	A \textbf{pseudo double category} $\dblcat{D}$, 
	or simply \textbf{double category}, consists of 
	a category of objects $\dblcat{D}_{0}$ and 
	a category of arrows $\dblcat{D}_{1}$ together
	with the following functors
	\begin{equation*}
		\begin{split}
			U & 
				\from \dblcat{D}_{0} \to \dblcat{D}_{1}, \\
			S,T & 
				\from \dblcat{D}_{1} \rightrightarrows \dblcat{D}_{0}, \t{ and} \\
			\odot & 
				\from \dblcat{D}_{1} \times_{\dblcat{D}_{0}} \dblcat{D}_{1} 
					\to \dblcat{D}_{1}
		\end{split}
	\end{equation*}
	where the pullback 
		$\dblcat{D}_{1} \times_{\dblcat{D}_{0}} \dblcat{D}_{1}$ 
	is taken over $S$ and $T$.  
	These functors satisfy the equations
	\begin{equation*}
		\begin{split}
			S(U_{A}) = A & = T(U_{A}) \\
			S(M \odot N) & = SN \\
			T(M \odot N) & = TM. 
	\end{split}
	\end{equation*}
	This also comes equipped with natural isomorphisms
	\begin{equation*}
		\begin{split}
		\alpha & \from (M \odot N) \odot P \to M \odot (N \odot P)\\
		\lambda & \from U_{B} \odot M \to M\\
		\rho & \from M \odot U_{A} \to M
	\end{split}
	\end{equation*}
	such that 
		$S(\alpha)$, 
		$S(\lambda)$, 
		$S(\rho)$, 
		$T(\alpha)$, 
		$T(\lambda)$, and 
		$T(\rho)$ 
	are each identities and that 
	the coherence axioms of a monoidal category are satisfied. 
	
	To match this definition with the more intuitive terms used, we say
	\textbf{vertical 1-morphisms} for the $\dblcat{D}_{0}$-morphisms,
	\textbf{horizontal 1-morphisms}\footnote{Sometimes the term \textbf{horizontal 1-cell} is used for these \cite{Shul}, and for good reason. A $(n \times 1)$-category consists of categories $\mathbf{D_i}$ for $0 \leq i \leq n$ where the objects of $\mathbf{D_i}$ are $i$-cells and the morphisms of $\mathbf{D_i}$ are vertical $i+1$-morphisms. A double category is then just a $(1 \times 1)$-category. From this perspective, `cells' are always objects with morphisms going between them.} for the $\dblcat{D}_{1}$-objects, and
	\textbf{2-morphisms} for the $\dblcat{D}_{1}$-morphisms. 
	As for notation, we write vertical and horizontal morphisms 
	with the arrows $\to$ and $\hto$, respectively, and 
	2-morphisms we draw as in 
		\eqref{diag:DblCatSquare}.
\end{defn}

An equivalent perspective to this definition 
is that a \emph{pseudo} double category is 
a category `weakly internal' to $\cat{Cat}$, whereas a category internal to $\mathbf{Cat}$ is an ordinary double category, meaning that the natural isomorphisms above are identities.

To bypass checking that a bicategory is monoidal, 
we instead need to check 
that a certain double category is monoidal. 
To define a monoidal double category, however, 
we need the notion of a \textbf{globular 2-morphism}.  
This is a 2-morphism whose 
source and target vertical 1-morphisms are identities.

%
\begin{defn}
	\label{def:MonoidalDoubleCategory}
	A \textbf{monoidal double category} is 
	a double category $\dblcat{D}$ 
	such that:
	\begin{enumerate}
		\item $\dblcat{D}_{0}$ and $\dblcat{D}_{1}$ 
			are both monoidal categories.
		\item If $I$ is the monoidal unit of $\dblcat{D}_{0}$, 
			then $U_I$ is the monoidal unit of $\dblcat{D}_{1}$.
		\item The functors $S$ and $T$ are strict monoidal and 
			preserve the associativity and unit constraints.
		\item There are globular 2-isomorphisms
			\[ 
				\mathfrak{x} \from 
					(M_1 \otimes N_1) \odot (M_2 \otimes N_2) 
					\to 
					(M_1\odot M_2) \otimes (N_1\odot N_2)
			\]
			and
			\[
				\mathfrak{u} \from 
				U_{A \otimes B} 
				\to 
				(U_A \otimes U_B)
			\]
			such that the following diagrams commute:
		\item \label{diag:MonDblCat}
			The following diagrams commute expressing the constraint data for the double functor $\otimes$.
			\[
			\begin{tikzpicture}[scale=0.9]
				\node (A) at (0,3) {\footnotesize{
							$((M_1\otimes N_1)\odot (M_2\otimes N_2)) \odot (M_3\otimes N_3)$}
				};
				\node (B) at (9,3) {\footnotesize{
						$((M_1\odot M_2)\otimes (N_1\odot N_2)) \odot (M_3\otimes N_3) $}
				};
				\node (A') at (0,1.5) {\footnotesize{
						$(M_1\otimes N_1)\odot ((M_2\otimes N_2) \odot (M_3\otimes N_3)) $}
				};
				\node (B') at (9,1.5) {\footnotesize{
						$((M_1\odot M_2)\odot M_3) \otimes ((N_1\odot N_2)\odot N_3)$}
				};
				\node (A'') at (0,0) {\footnotesize{
						$(M_1\otimes N_1) \odot ((M_2\odot M_3) \otimes (N_2\odot N_3))$}
				};
				\node (B'') at (9,0) {\footnotesize{
						$(M_1\odot (M_2\odot M_3)) \otimes (N_1\odot (N_2\odot N_3))$}
				};
			\path[->,font=\scriptsize]
				(A) edge node[left]{$\alpha$} (A')
				(A') edge node[left]{$1 \odot \mathfrak{x}$} (A'')
				(B) edge node[right]{$\mathfrak{x}$} (B')
				(B') edge node[right]{$\alpha \otimes \alpha$} (B'')
				(A) edge node[above]{$\mathfrak{x} \odot 1$} (B)
				(A'') edge node[above]{$\mathfrak{x}$} (B'');
		\end{tikzpicture}
		\]
		\[
		\begin{tikzpicture}[scale=0.9]
			\node (UL) at (0,1.5) {\footnotesize{
					$(M\otimes N) \odot U_{C\otimes D}$}
			};
			\node (LL) at (0,0) {\footnotesize{
					$M\otimes N$}
			};
			\node (UR) at (4.5,1.5) {\footnotesize{
					$(M\otimes N)\odot (U_C\otimes U_D)$}
			};
			\node (LR) at (4.5,0) {\footnotesize{
					$(M\odot U_C) \otimes (N\odot U_D)$}
			};
			\path[->,font=\scriptsize]
				(UL) edge node[above]{$1 \odot \mathfrak{u}$} (UR) 
				(UL) edge node[left]{$\rho$} (LL)
				(LR) edge node[above]{$\rho \otimes \rho$} (LL)
				(UR) edge node[right]{$\mathfrak{x}$} (LR);
		\end{tikzpicture}
		%
		%
		\begin{tikzpicture}[scale=0.9]
			\node (UL) at (0,1.5) {\scriptsize{$U_{A\otimes B}\odot (M\otimes N)$}};
			\node (LL) at (0,0) {\scriptsize{$M\otimes N$}};
			\node (UR) at (4.5,1.5) {\scriptsize{$(U_A\otimes U_B)\odot (M\otimes N)$}};
			\node (LR) at (4.5,0) {\scriptsize{$(U_A \odot M) \otimes (U_B\odot N)$}};
			\path[->,font=\scriptsize]
				(UL) edge node[above]{$\mathfrak{u} \odot 1$} (UR) 
				(UL) edge node[left]{$\lambda$} (LL)
				(LR) edge node[above]{$\lambda \otimes \lambda$} (LL)
				(UR) edge node[right]{$\mathfrak{x}$} (LR);
		\end{tikzpicture}
		\]
		\item The following diagrams commute expressing 
		the associativity isomorphism for $\otimes$ is a transformation of double categories.
		\[
		\begin{tikzpicture}[scale=0.9]
			\node (A) at (0,3) {\footnotesize{
					$((M_1\otimes N_1)\otimes P_1) \odot ((M_2\otimes N_2)\otimes P_2)$}
			};
			\node (B) at (9,3) {\footnotesize{
					$(M_1\otimes (N_1\otimes P_1)) \odot (M_2\otimes (N_2\otimes P_2))$}
			};
			\node (A') at (0,1.5) {\footnotesize{
					$((M_1\otimes N_1) \odot (M_2\otimes N_2)) \otimes (P_1\odot P_2)$}
			};
			\node (B') at (9,1.5) {\footnotesize{
					$(M_1\odot M_2) \otimes ((N_1\otimes P_1)\odot (N_2\otimes P_2))$}
			};
			\node (A'') at (0,0) {\footnotesize{
					$((M_1\odot M_2) \otimes(N_1\odot N_2)) \otimes (P_1\odot P_2)$}
			};
			\node (B'') at (9,0) {\footnotesize{
					$(M_1\odot M_2) \otimes ((N_1\odot N_2)\otimes (P_1\odot P_2))$}
			};
			\path[->,font=\scriptsize]
				(A) edge node[left]{$\mathfrak{x}$} (A')
				(A') edge node[left]{$\mathfrak{x} \otimes 1$} (A'')
				(B) edge node[right]{$\mathfrak{x}$} (B')
				(B') edge node[right]{$1 \otimes \mathfrak{x}$} (B'')
				(A) edge node[above]{$a \odot a$} (B)
				(A'') edge node[above]{$a$} (B'');
		\end{tikzpicture}
		\]
		\[
		\begin{tikzpicture}
			\node (A) at (0,3) {\footnotesize{$U_{(A\otimes B)\otimes C}$}};
			\node (B) at (4,3) {\footnotesize{$U_{A\otimes (B\otimes C)} $}};
			\node (A') at (0,1.5) {\footnotesize{$U_{A\otimes B} \otimes U_C $}};
			\node (B') at (4,1.5) {\footnotesize{$U_A\otimes U_{B\otimes C}$}};
			\node (A'') at (0,0) {\footnotesize{$(U_A\otimes U_B)\otimes U_C$}};
			\node (B'') at (4,0) {\footnotesize{$U_A\otimes (U_B\otimes U_C) $}};
			\path[->,font=\scriptsize]
				(A) edge node[left]{$\mathfrak{u}$} (A')
				(A') edge node[left]{$\mathfrak{u} \otimes 1$} (A'')
				(B) edge node[right]{$\mathfrak{u}$} (B')
				(B') edge node[right]{$1 \otimes \mathfrak{u}$} (B'')
				(A) edge node[above]{$U_{a}$} (B)
				(A'') edge node[above]{$a$} (B'');
		\end{tikzpicture}
		\]
		\item The following diagrams commute expressing that 
		the unit isomorphisms for $\otimes$ are transformations of double categories. 
		\[
		\begin{tikzpicture}
			\node (A) at (0,1.5) {\footnotesize{$(M\otimes U_I)\odot (N\otimes U_I)$}};
			\node (A') at (0,0) {\footnotesize{$M\odot N $}};
			\node (B) at (4,1.5) {\footnotesize{$(M\odot N)\otimes (U_I \odot U_I) $}};
			\node (B') at (4,0) {\footnotesize{$(M\odot N)\otimes U_I $}};
			\path[->,font=\scriptsize]
				(A) edge node[left]{$r \odot r$} (A')
				(A) edge node[above]{$\mathfrak{x}$} (B)
				(B) edge node[right]{$1 \otimes \rho$} (B')
				(B') edge node[above]{$r$} (A');
		\end{tikzpicture}
		\quad
		\begin{tikzpicture}
			\node (A) at (0,0.75) {\footnotesize{$U_{A\otimes I} $}};
			\node (B) at (1.5,1.5) {\footnotesize{$U_A\otimes U_I $}};
			\node (B') at (1.5,0) {\footnotesize{$U_A$}};
			\path[->,font=\scriptsize]
				(A) edge node[above]{$\mathfrak{u}$} (B)
				(A) edge node[below]{$U_{r}$} (B')
				(B) edge node[right]{$r$} (B');
		\end{tikzpicture}
		\]
		\[
		\begin{tikzpicture}
			\node (A) at (0,1.5) {\footnotesize{$(U_I\otimes M)\odot (U_I\otimes N)$}};
			\node (A') at (0,0) {\footnotesize{$M\odot N$}};
			\node (B) at (4,1.5) {\footnotesize{$(U_I \odot U_I) \otimes (M\odot N)$}};
			\node (B') at (4,0) {\footnotesize{$U_I\otimes (M\odot N) $}};
			\path[->,font=\scriptsize]
				(A) edge node[left]{$\ell \odot \ell$} (A')
				(A) edge node[above]{$\mathfrak{x}$} (B)
				(B) edge node[right]{$\lambda \otimes 1$} (B')
				(B') edge node[above]{$\ell$} (A');
		\end{tikzpicture}
		\quad
		\begin{tikzpicture}
			\node (A) at (0,0.75) {\footnotesize{$U_{I\otimes A}$}};
			\node (B) at (1.5,1.5) {\footnotesize{$U_I\otimes U_A$}};
			\node (B') at (1.5,0) {\footnotesize{$U_A$}};
			\path[->,font=\scriptsize]
				(A) edge node[above]{$\mathfrak{u}$} (B)
				(A) edge node[below]{$U_{\ell}$} (B')
				(B) edge node[right]{$\ell$} (B');
		\end{tikzpicture}
		\]
		\newcounter{mondbl}
		\setcounter{mondbl}{\value{enumi}}
	\end{enumerate}
	A \textbf{braided monoidal double category} 
	is a monoidal double category 
	such that:
	\begin{enumerate}
		\setcounter{enumi}{\value{mondbl}}
		\item $\dblcat{D}_{0}$ and $\dblcat{D}_{1}$ are braided monoidal categories.
		\item The functors $S$ and $T$ are strict braided monoidal functors.
		\item The following diagrams commute expressing that the braiding is a transformation of double categories.
		\[
		\begin{tikzpicture}
			\node (A) at (0,1.5) {\footnotesize{$(M_1 \odot M_2) \otimes (N_1 \odot N_2)$}};
			\node (A') at (0,0) {\footnotesize{$(M_1\otimes N_1) \odot (M_2\otimes N_2)$}};
			\node (B) at (5,1.5) {\footnotesize{$(N_1\odot N_2) \otimes (M_1 \odot M_2)$}};
			\node (B') at (5,0) {\footnotesize{$(N_1 \otimes M_1) \odot (N_2 \otimes M_2)$}};
			\path[->,font=\scriptsize]
				(A) edge node[left]{$\mathfrak{x}$} (A')
				(A) edge node[above]{$\mathfrak{s}$} (B)
				(B) edge node[right]{$\mathfrak{x}$} (B')
				(A') edge node[above]{$\mathfrak{s} \odot \mathfrak{s}$} (B');
		\end{tikzpicture}
		\quad
		\begin{tikzpicture}
			\node (A) at (0,1.5) {\footnotesize{$U_A \otimes U_B$}};
			\node (A') at (0,0) {\footnotesize{$U_B\otimes U_A$}};
			\node (B) at (2,1.5) {\footnotesize{$U_{A\otimes B} $}};
			\node (B') at (2,0) {\footnotesize{$U_{B\otimes A}$}};
			\path[->,font=\scriptsize]
				(A) edge node[left]{$\mathfrak{s}$} (A')
				(B) edge node[above]{$\mathfrak{u}$} (A)
				(B) edge node[right]{$U_\mathfrak{s}$} (B')
				(B') edge node[above]{$\mathfrak{u}$} (A');
		\end{tikzpicture}
		\]
		\setcounter{mondbl}{\value{enumi}}
	\end{enumerate}
	Finally, a \textbf{symmetric monoidal double category} 
	is a braided monoidal double category $\mathbb{D}$ such that
	\begin{enumerate}
		\setcounter{enumi}{\value{mondbl}}
		\item $\dblcat{D}_{0}$ and $\dblcat{D}_{1}$ are symmetric monoidal.
	\end{enumerate}
\end{defn}

%
\begin{defn}
	\label{def:CompanionConjoint}
	Let $\mathbb{D}$ be a double category and 
	$f \from A\to B$ a vertical 1-morphism.  
	A \textbf{companion} of $f$ is a horizontal 1-morphism
		$\widehat{f} \from A \hto B$ 
	together with 2-morphisms
	\[
	\raisebox{-0.5\height}{
	\begin{tikzpicture}
		\node (A) at (0,1) {$A$};
		\node (B) at (1,1) {$B$};
		\node (A') at (0,0) {$B$};
		\node (B') at (1,0) {$B$};
		\path[->,font=\scriptsize,>=angle 90]
			(A) edge node[above]{$\widehat{f}$} (B)
			(A) edge node[left]{$f$} (A')
			(B) edge node[right]{$B$} (B')
			(A') edge node[below]{$U_B$} (B');
		\draw (0.5,.925) -- (0.5,1.075);
		\draw (0.5,-.075) -- (0.5,.075);
		\node () at (0.5,0.5) {\scriptsize{$\Downarrow$}};
	\end{tikzpicture}
	}
	\quad \text{ and } \quad
	\raisebox{-0.5\height}{
	\begin{tikzpicture}
		\node (A) at (0,1) {$A$};
		\node (B) at (1,1) {$A$};
		\node (A') at (0,0) {$A$};
		\node (B') at (1,0) {$B$};
		\path[->,font=\scriptsize,>=angle 90]
			(A) edge node[above]{$U_A$} (B)
			(A) edge node[left]{$A$} (A')
			(B) edge node[right]{$f$} (B')
			(A') edge node[below]{ $\widehat{f}$} (B');
		\draw (0.5,.925) -- (0.5,1.075);
		\draw (0.5,-.075) -- (0.5,.075);
		\node () at (0.5,0.5) {\scriptsize{$\Downarrow$}};
	\end{tikzpicture}
	}
	\]
	such that the following equations hold:
	\begin{equation}
	\label{eq:CompanionEq}
	\raisebox{-0.5\height}{
	\begin{tikzpicture}
		\node (A) at (0,2) {$A$};
		\node (B) at (1.1,2) {$A$};
		\node (A') at (0,1) {$A$};
		\node (B') at (1.1,1) {$B$};
		\node (A'') at (0,0) {$B$};
		\node (B'') at (1.1,0) {$B$};
		\path[->,font=\scriptsize,>=angle 90]
			(A) edge node[left]{$A$} (A')
			(A') edge node[left]{$f$} (A'')
			(B) edge node[right]{$f$} (B')
			(B') edge node[right]{$B$} (B'')
			(A) edge node[above]{$U_A$} (B)
			(A') edge  (B')
			(A'') edge node[below]{$U_B$} (B'');
		\draw (0.5,1.925) -- (0.5,2.075);
		\draw[line width=2mm,white] (0.5,.925) -- (0.5,1.075);
		\draw (0.5,-.075) -- (0.5,.075);
		\node () at (0.5,0.5) {\scriptsize{$\Downarrow$}};
		\node () at (0.5,1.5) {\scriptsize{$\Downarrow$}};
		\node () at (0.5,1) {\scriptsize $\widehat{f}$};
	\end{tikzpicture}
	}
	\raisebox{-0.5\height}{=}
	\raisebox{-0.5\height}{
	\begin{tikzpicture}
		\node (A) at (0,1) {$A$};
		\node (B) at (1,1) {$A$};
		\node (A') at (0,0) {$B$};
		\node (B') at (1,0) {$B$};
		\path[->,font=\scriptsize,>=angle 90]
		(A) edge node[left]{$f$} (A')
		(B) edge node[right]{$f$} (B')
		(A) edge node[above]{$U_A$} (B)
		(A') edge node[below]{$U_B$} (B');
		\draw (0.5,.925) -- (0.5,1.075);
		\draw (0.5,-.075) -- (0.5,.075);
		\node () at (0.5,0.5) {\scriptsize{$\Downarrow U_f$}};
	\end{tikzpicture}
	}
	\raisebox{-0.5\height}{\text{   and   }}
	\raisebox{-0.5\height}{
	\begin{tikzpicture}
		\node (A) at (0,1) {$A$};
		\node (A') at (0,0) {$A$};
		\node (B) at (1,1) {$A$};
		\node (B') at (1,0) {$B$};
		\node (C) at (2,1) {$B$};
		\node (C') at (2,0) {$B$};
		\path[->,font=\scriptsize,>=angle 90]
			(A) edge node[left]{$A$} (A')
			(B) edge node[left]{$f$} (B')
			(C) edge node[right]{$B$} (C')
			(A) edge node[above]{$U_A$} (B)
			(B) edge node[above]{$\widehat{f}$} (C)
			(A') edge node[below]{$\widehat{f}$} (B')
			(B') edge node[below]{$U_B$} (C');
		\draw (1.5,0.925) -- (1.5,1.075);
		\draw (1.5,0.925) -- (1.5,1.075);
		\draw (0.5,.925) -- (0.5,1.075);
		\draw (0.5,-.075) -- (0.5,.075);
		\node () at (0.5,0.5) {\scriptsize{$\Downarrow$}};
		\node () at (1.5,0.5) {\scriptsize{$\Downarrow$}};
	\end{tikzpicture}
	}
	\raisebox{-0.5\height}{=}
	\raisebox{-0.5\height}{
	\begin{tikzpicture}
		\node (A) at (0,1) {$A$};
		\node (B) at (1,1) {$B$};
		\node (A') at (0,0) {$A$};
		\node (B') at (1,0) {$B$};
		\path[->,font=\scriptsize,>=angle 90]
			(A) edge node[left]{$A$} (A')
			(B) edge node[right]{$B$} (B')
			(A) edge node[above]{$\widehat{f}$} (B)
			(A') edge node[below]{$\widehat{f}$} (B');
		\draw (0.5,.925) -- (0.5,1.075);
		\draw (0.5,-.075) -- (0.5,.075);
		\node () at (0.5,0.5) {\scriptsize{$\Downarrow \id_{\widehat{f}}$}};
	\end{tikzpicture}
	}
	\end{equation}
	A \textbf{conjoint} of $f$, denoted 
		$\check{f} \from B \hto A$, 
	is a companion of $f$ in the double category 
		$\dblcat{D}^{h\cdot\mathrm{op}}$ 
	obtained by reversing the horizontal 1-morphisms, 
	but \emph{not} the vertical 1-morphisms.
\end{defn}

%
\begin{defn}
	\label{def:Fibrant}
	We say that a double category is \textbf{fibrant} 
	if every vertical 1-morphism has 
	both a companion and a conjoint. 
	If every \emph{invertible} vertical 1-morphism 
	has both a companion and a conjoint, 
	then we say the double category is \textbf{isofibrant}.
\end{defn}

The final piece we need to present 
the main theorem of this section is the following.  
Given a double category $\dblcat{D}$, 
the \textbf{horizontal edge bicategory} 
	$H(\dblcat{D})$ 
of $\dblcat{D}$ is the bicategory whose 
objects are those of $\dblcat{D}$, 
morphisms are horizontal 1-morphisms of $\dblcat{D}$, 
and $2$-morphisms are the globular 2-morphisms.

\begin{thm}[Shulman {\cite[Theorem 5.1]{Shul}}]
	\label{thm:DoubleGivesBi}
	Let $\dblcat{D}$ be an isofibrant symmetric monoidal double category. 
	Then $H(\dblcat{D})$ is a symmetric monoidal bicategory.  
\end{thm}

Thanks to Theorem \ref{thm:DoubleGivesBi}, 
we can show that $\bimonspcsp{T}$
is symmetric monoidal
much more efficiently than
if we were to 
drudge through
all of the axioms.
But before we do this,
we recall the notion of compactness
in a bicategory.

\subsection{Duality in bicategories}
\label{sec:CompactClosed}

In this section, we introduce various 
notions of duality in order to define
`compact closed bicategories' 
as conceived by Stay \cite{Stay}. 
We write $LR$ for the tensor product 
of objects $L$ and $R$ and 
$fg$ for the tensor product 
of morphisms $f$ and $g$.  
This lets us reserve the symbol `$\otimes$' 
for the horizontal composition functor 
of a bicategory.

\begin{defn}
	\label{def:DualPairCat}
	A \textbf{dual pair} in a monoidal category 
	is a tuple $(L,R,e,c)$ with objects $L$ and $R$, 
	called the \textbf{left} and \textbf{right} duals, 
	and morphisms
	\[
		e \from LR \to I 
		\quad \quad 
		c \from I \to RL,
	\]
	called the \textbf{counit} and \textbf{unit}, respectively, such that the following diagams commute.
	\[
	\begin{tikzpicture}
		\node (L1) at (0,1) {$L$};
		\node (L2) at (0,0) {$L$};
		\node (LRL) at (1.5,0.5) {$LRL$};
		\path[->,font=\scriptsize,>=angle 90]
		(L1) edge node[left]{$L$} (L2)
		(L1) edge node[above]{$Lc$} (LRL)
		(LRL) edge node[below]{$eL$} (L2);
	\end{tikzpicture}
	\quad \quad
	\begin{tikzpicture}
		\node (R1) at (0,1) {$R$};
		\node (R2) at (0,0) {$R$};
		\node (RLR) at (1.5,0.5) {$RLR$};
		\path[->,font=\scriptsize,>=angle 90]
		(R1) edge node[left]{$R$} (R2)
		(R1) edge node[above]{$cR$} (RLR)
		(RLR) edge node[below]{$Re$} (R2);
	\end{tikzpicture}	
	\]
\end{defn}

\begin{defn}
	\label{def:DualPairBicat}
	Inside a monoidal bicategory, a  
	\textbf{dual pair} is a tuple 
	$(L,R,e,c,\alpha,\beta)$ with 
	objects $L$ and $R$, 
	morphisms
	\[
		e \from LR \to I 
		\quad \quad 
		c \from I \to RL,
	\]
	and invertible 2-morphisms
	\[
	\begin{tikzpicture}
		\node (L1) at (0,2) {$L$};
		\node (LI) at (0,1) {$LI$};
		\node (LRL1) at (0,0) {$L(RL)$};
		\node (LRL2) at (2,0) {$(LR)L$};
		\node (IL) at (2,1) {$IL$};
		\node (L2) at (2,2) {$L$};
		\path[->,font=\scriptsize,>=angle 90]
			(L1) edge (LI)
			(LI) edge node[left] {$Lc$} (LRL1)
			(LRL1) edge (LRL2)
			(LRL2) edge node[right] {$eL$} (IL)
			(IL) edge (L2)
			(L1) edge node[above] {$L$} (L2);
		\node () at (1,1) {$\Downarrow \alpha$};
	\end{tikzpicture}
	\quad \quad
	\begin{tikzpicture}
		\node (L1) at (0,2) {$R$};
		\node (LI) at (0,1) {$RI$};
		\node (LRL1) at (0,0) {$(RL)R$};
		\node (LRL2) at (2,0) {$R(LR)$};
		\node (IL) at (2,1) {$RI$};
		\node (L2) at (2,2) {$R$};
		\path[->,font=\scriptsize,>=angle 90]
			(L1) edge (LI)
			(LI) edge node[left] {$cR$} (LRL1)
			(LRL1) edge (LRL2)
			(LRL2) edge node[right] {$Re$} (IL)
			(IL) edge (L2)
			(L1) edge node[above] {$R$} (L2);
		\node () at (1,1) {$\Downarrow \beta$};
	\end{tikzpicture}
	\]
	called \textbf{cusp isomorphisms}.  
	If this data satisfies the swallowtail equations 
	in the sense that the diagrams in Figure 
		\ref{fig:Swallowtail} 
	are identities, 
	then we call the dual pair \textbf{coherent}.
\end{defn}


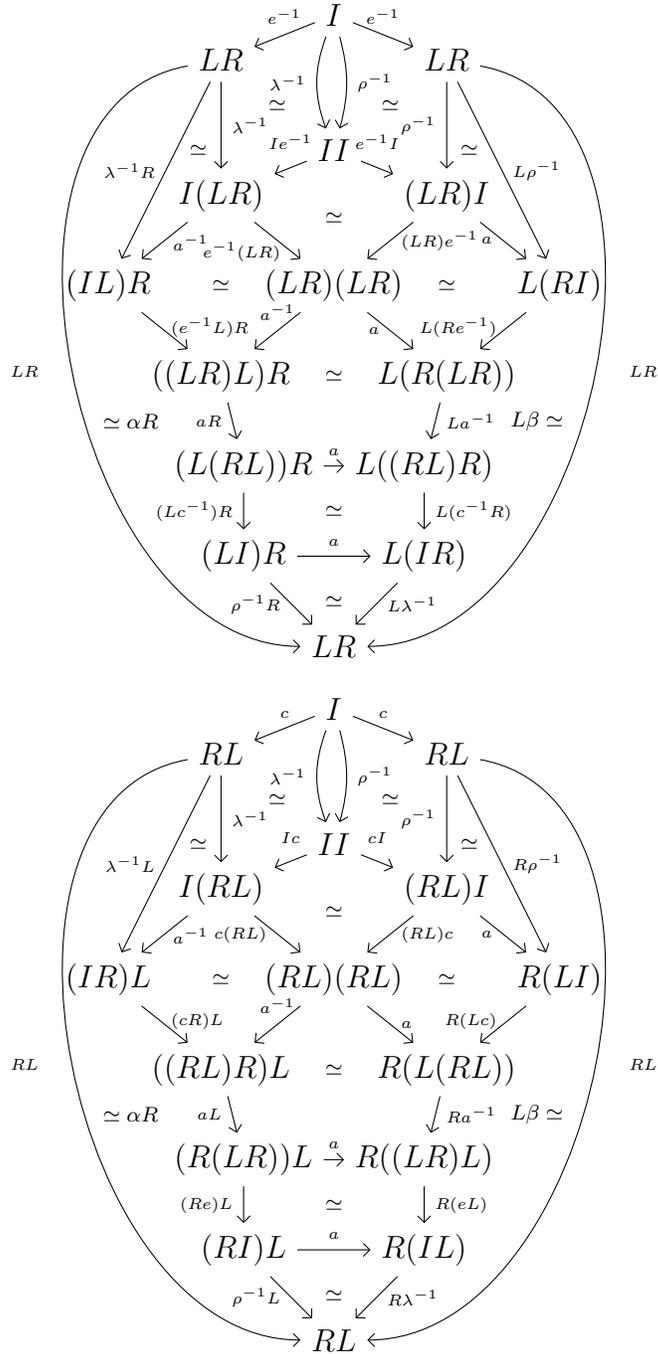
\begin{figure}
	%
	%
	\[
	\begin{tikzpicture}[scale=0.60]
	\node (A) at (0,0) {$I$};
	\node (B) at (-2.5,-1) {$L R$};
	\node (C) at (2.5,-1) {$L R$};
	\node (D) at (0,-3) {$I I$};
	\node (E) at (-2.5,-4) {$I (L R)$};
	\node (F) at (2.5,-4) {$(L R) I$};
	\node (G) at (0,-6) {$(L R) (L R)$};
	\node (H) at (-5,-6) {$(I L) R$};
	\node (I) at (5,-6) {$L (R I)$};
	\node (J) at (-2.5,-8) {$((L R) L)  R$};
	\node (K) at (2.5,-8) {$L (R (L R))$};
	\node (L) at (-2,-10) {$(L (R L)) R$};
	\node (M) at (2,-10) {$L  ((R L) R)$};
	\node (N) at (-2,-12) {$(L I) R$};
	\node (O) at (2,-12) {$L (I R)$};
	\node (P) at (0,-14) {$L  R$};
	%
	%
	\node (Q) at (-1.25,-2) {\scriptsize{$\simeq$}};
	\node (R) at (1.25,-2) {\scriptsize{$\simeq$}};
	\node (S) at (0,-4.5) {\scriptsize{$\simeq$}};
	\node (T) at (-3,-3) {\scriptsize{$\simeq $}};
	\node (U) at (3,-3) {\scriptsize{$\simeq$}};
	\node (V) at (-2.5,-6) {\scriptsize{$\simeq$}};
	\node (W) at (2.5,-6) {\scriptsize{$\simeq$}};
	\node (X) at (0,-8) {\scriptsize{$\simeq$}};
	\node (Y) at (0,-11) {\scriptsize{$\simeq$}};
	\node (Z) at (0,-13) {\scriptsize{$\simeq$}};
	\node (A1) at (-4.5,-9) {\scriptsize{$\simeq \alpha R$}};
	\node (A2) at (4.5,-9) {\scriptsize{$L \beta \simeq$}};
	\path[->,font=\tiny,>=angle 90]
	(A) edge node[above]{$e^{-1}$} (B)
	(A) edge node[above]{$e^{-1}$} (C)
	(A) edge[out=-110,in=110] node[left]{$\lambda^{-1}$} (D)
	(A) edge[out=-75,in=75] node[right]{$\rho^{-1}$} (D)
	(B) edge node[right]{$\lambda^{-1}$}(E)
	(C) edge node[left]{$\rho^{-1}$} (F)
	(D) edge node[left=0.2cm,above=0.1cm]{$I e^{-1}$} (E)
	(D) edge node[left=0.2cm,above=0.1cm]{$e^{-1} I$} (F)
	(E) edge node[below,left,pos=0.8]{$e^{-1} (L R)$} (G)
	(F) edge node[below,right]{$(L R) e^{-1}$} (G)
	(B) edge node[above,left]{$\lambda^{-1} R$} (H)
	(C) edge node[above,right]{$L \rho^{-1}$} (I)
	(E) edge node[below,right,pos=0.55]{$a^{-1}$} (H)
	(F) edge node[below,left]{$a$} (I)
	(H) edge node[above,right,pos=0.4]{$(e^{-1} L) R$} (J)
	(G) edge node[above]{$a^{-1}$} (J)
	(I) edge node[above,left,pos=0.4]{$L (R e^{-1})$} (K)
	(G) edge node[below,left]{$a$} (K)
	(J) edge node[left]{$a R$} (L)
	(K) edge node[right]{$L a^{-1}$} (M)
	(L) edge node[above]{$a$} (M)
	(L) edge node[left]{$(L  c^{-1}) R$} (N)
	(M) edge node[right]{$L (c^{-1} R)$} (O)
	(N) edge node[above]{$a$} (O)
	(N) edge node[below,left]{$\rho^{-1} R$} (P)
	(O) edge node[below,right]{$L \lambda^{-1}$} (P)
	(B) edge [out=-170,in=-180] node [below=0.3cm,left=0.3cm] {$L R$} (P)
	(C) edge [in=-360,out=-370,] node [below=0.3cm,right=0.3cm] {$LR$} (P);
	\end{tikzpicture}
	\]
	%
	%
	\[
	\begin{tikzpicture}[scale=0.60]
	\node (A) at (0,0) {$I$};
	\node (B) at (-2.5,-1) {$RL$};
	\node (C) at (2.5,-1) {$RL$};
	\node (D) at (0,-3) {$I I$};
	\node (E) at (-2.5,-4) {$I (RL)$};
	\node (F) at (2.5,-4) {$(RL) I$};
	\node (G) at (0,-6) {$(RL) (RL)$};
	\node (H) at (-5,-6) {$(I R) L$};
	\node (I) at (5,-6) {$R (L I)$};
	\node (J) at (-2.5,-8) {$((RL) R)  L$};
	\node (K) at (2.5,-8) {$R (L (R L))$};
	\node (L) at (-2,-10) {$(R (L R)) L$};
	\node (M) at (2,-10) {$R  ((L R) L)$};
	\node (N) at (-2,-12) {$(R I) L$};
	\node (O) at (2,-12) {$R (I L)$};
	\node (P) at (0,-14) {$R  L$};
	%
	%
	\node (Q) at (-1.25,-2) {\scriptsize{$\simeq$}};
	\node (R) at (1.25,-2) {\scriptsize{$\simeq$}};
	\node (S) at (0,-4.5) {\scriptsize{$\simeq$}};
	\node (T) at (-3,-3) {\scriptsize{$\simeq $}};
	\node (U) at (3,-3) {\scriptsize{$\simeq$}};
	\node (V) at (-2.5,-6) {\scriptsize{$\simeq$}};
	\node (W) at (2.5,-6) {\scriptsize{$\simeq$}};
	\node (X) at (0,-8) {\scriptsize{$\simeq$}};
	\node (Y) at (0,-11) {\scriptsize{$\simeq$}};
	\node (Z) at (0,-13) {\scriptsize{$\simeq$}};
	\node (A1) at (-4.5,-9) {\scriptsize{$\simeq \alpha R$}};
	\node (A2) at (4.5,-9) {\scriptsize{$L \beta \simeq$}};
	\path[->,font=\tiny,>=angle 90]
	(A) edge[->] node[above]{$c$} (B)
	(A) edge[->] node[above]{$c$} (C)
	(A) edge[->,out=-110,in=110] node[left]{$\lambda^{-1}$} (D)
	(A) edge[->,out=-75,in=75] node[right]{$\rho^{-1}$} (D)
	(B) edge[->] node[right]{$\lambda^{-1}$}(E)
	(C) edge[->] node[left]{$\rho^{-1}$} (F)
	(D) edge[->] node[left=0.2cm,above=0.1cm]{$I c$} (E)
	(D) edge[->] node[left=0.2cm,above=0.1cm]{$c I$} (F)
	(E) edge[->] node[below,left,pos=0.5]{$c (R L)$} (G)
	(F) edge[->] node[below,right]{$(R L) c$} (G)
	(B) edge[->] node[above,left]{$\lambda^{-1} L$} (H)
	(C) edge[->] node[above,right]{$R \rho^{-1}$} (I)
	(E) edge[->] node[below,right,pos=0.55]{$a^{-1}$} (H)
	(F) edge[->] node[below,left]{$a$} (I)
	(H) edge[->] node[above,right,pos=0.4]{$(c R) L$} (J)
	(G) edge[->] node[above]{$a^{-1}$} (J)
	(I) edge[->] node[above,left,pos=0.4]{$R (L c)$} (K)
	(G) edge[->] node[above,right]{$a$} (K)
	(J) edge[->] node[left]{$a L$} (L)
	(K) edge[->] node[right]{$R a^{-1}$} (M)
	(L) edge[->] node[above]{$a$} (M)
	(L) edge[->] node[left]{$(R  e) L$} (N)
	(M) edge[->] node[right]{$R (e L)$} (O)
	(N) edge[->] node[above]{$a$} (O)
	(N) edge[->] node[below,left]{$\rho^{-1} L$} (P)
	(O) edge[->] node[below,right]{$R \lambda^{-1}$} (P)
	(B) edge[->,out=-170,in=-180] node [below=0.3cm,left=0.3cm] {$R L$} (P)
	(C) edge[->,in=-360,out=-370] node [below=0.3cm,right=0.3cm] {$RL$} (P);
	\end{tikzpicture}
	\]
	\caption{The swallowtail diagrams for the (co)unit.}
	\label{fig:Swallowtail}
\end{figure}

Recall that a symmetric monoidal category is called 
\textbf{compact closed} if every object is part of a dual pair. 
We can generalize this idea to bicategories by 
introducing 2-morphisms and some coherence axioms. 
The following definition is due to Stay 
	\cite{Stay}.

\begin{defn}
	\label{def:CompClosdBicat}
	A \textbf{compact closed} bicategory is a symmetric monoidal bicategory for which
	every object $R$ 
	is part of a coherent dual pair. 
\end{defn}

The difference between showing compact closededness 
in categories versus bicategories might seem quite large 
because of the swallowtail equations.  
Looking at Figure 
	\ref{fig:Swallowtail}, 
it is no surprise that these can be incredibly tedious to work with.  
Fortunately, Pstr\k{a}gowski \cite{Piotr} proved a wonderful strictification theorem 
that effectively circumvents the need to consider the swallowtail equations.  

\begin{thm}[{\cite[p.~22]{Piotr}}]
	\label{thm:StrictingDualPairs}
	Given a dual pair $(L,R,e,c,\alpha,\beta)$, 
	we can find a cusp isomorphism $\beta'$ such that
	 $(L,R,e,c,\alpha,\beta')$ is a coherent dual pair.
\end{thm}

With the requisite background covered,
we can move on to our main results.

\section{Main results} 
\label{sec:SpansCospans}

\begin{defn}
\label{def:DblCatMonSpanCsp}
	Let $\cat{T}$ be a topos. 
	We will define a double category 
		$\dblmonspcsp{T}$ 
	whose objects are the objects of $\mathbf{T}$,
	vertical 1-morphisms are isomorphism classes of spans with invertible legs in $\mathbf{T}$, 
	horizontal 1-morphisms are cospans in $\mathbf{T}$, and 
	2-morphisms are isomorphism classes of spans of cospans in $\cat{T}$ with monic legs.
	In other words, a $2$-morphism is a
	commuting diagram in $\mathbf{T}$ of the form:
	\[
	\begin{tikzpicture}
	\node (A) at (0,2) {$\bullet$};
	\node (A') at (0,1) {$\bullet$};
	\node (A'') at (0,0) {$\bullet$};
	\node (B) at (1,2) {$\bullet$};
	\node (B') at (1,1) {$\bullet$};
	\node (B'') at (1,0) {$\bullet$};
	\node (C) at (2,2) {$\bullet$};
	\node (C') at (2,1) {$\bullet$};
	\node (C'') at (2,0) {$\bullet$};
	\path[->,font=\scriptsize,>=angle 90]
	(A) edge node[above]{$ $} (B)
	(A') edge node[above]{$ $} (B')
	(A'') edge node[above]{$ $} (B'')
	(C) edge node[above]{$ $} (B)
	(C') edge node[above]{$ $} (B')
	(C'') edge node[above]{$ $} (B'')
	(A') edge node[left]{$\cong$} (A)
	(A') edge node[left]{$\cong$} (A'')
	(B') edge[>->] node[left]{$ $} (B)
	(B') edge[>->] node[left]{$ $} (B'')
	(C') edge node[right]{$\cong$} (C)
	(C') edge node[right]{$\cong$} (C'');
	\end{tikzpicture}
	\]
\end{defn}

\begin{lem}
\label{lem:SpanCospanDoubleCat}
	$\dblmonspcsp{T}$ is a double category.  
\end{lem}

\begin{proof}
	Denote $\dblmonspcsp{T}$ by $\dblcat{M}$.  
	The object category $\dblcat{M}_0$ is given by objects of $\mathbf{T}$ and isomorphism classes of spans in $\mathbf{T}$ such that each leg is an isomorphism. 
	The arrow category $\dblcat{M}_1$ has 
	as objects the cospans in $\mathbf{T}$ and 
	as morphisms the isomorphism classes of spans of cospans with monic legs 
	as in the diagram above. We denote a span 
$x \gets y \to z$ as 
$y \from x \tospan z$ and a cospan 
$x \to y \gets z$ as 
$y \from x \tocospan z$.
	
	The functor 
		$U \from \dblcat{M}_0 \to \dblcat{M}_1$, introduced in \ref{def:DoubleCategory}, 
	acts on objects by mapping $x$ to the identity cospan on $x$ and 
	on morphisms by mapping $y \from x \tospan z$, 
	whose legs are isomorphisms, 
	to the square
	\[
	\begin{tikzpicture}
		\node (A) at (0,2) {$x$};
		\node (A') at (0,1) {$y$};
		\node (A'') at (0,0) {$z$};
		\node (B) at (1,2) {$x$};
		\node (B') at (1,1) {$y$};
		\node (B'') at (1,0) {$z$};
		\node (C) at (2,2) {$x$};
		\node (C') at (2,1) {$y$};
		\node (C'') at (2,0) {$z$};
		\path[->,font=\scriptsize,>=angle 90]
		(A) edge node{} (B)
		(A') edge node{} (B')
		(A'') edge node{} (B'')
		(C) edge node{} (B)
		(C') edge node{} (B')
		(C'') edge node{} (B'')
		(A') edge node{} (A)
		(A') edge node{} (A'')
		(B') edge node{} (B)
		(B') edge node{} (B'')
		(C') edge node{} (C)
		(C') edge node{} (C'');
	\end{tikzpicture}
	\]
	The functor 
		$S \from \dblcat{M}_1 \to \dblcat{M}_0$, also introduced in \ref{def:DoubleCategory},
	acts on objects by sending 
		$y \from x \tocospan z$ 
	to $x$ and on morphisms by sending 
	a square to the span occupying the  square's left vertical side.  
	The other functor $T$ is defined similarly.  
	
	The horizontal composition functor 
		$\odot \from \dblcat{M}_1 \times_{\dblcat{M}_0} \dblcat{M}_1 \to \dblcat{M}_1$ 
	acts on objects by composing cospans with pushouts in the usual way.  
	It acts on morphisms by 
	\[
	\raisebox{-0.5\height}{
		\begin{tikzpicture}
		\node (A) at (0,2) {$a$};
		\node (A') at (0,1) {$a'$};
		\node (A'') at (0,0) {$a''$};
		\node (B) at (1,2) {$b$};
		\node (B') at (1,1) {$b'$};
		\node (B'') at (1,0) {$b''$};
		\node (C) at (2,2) {$c$};
		\node (C') at (2,1) {$c'$};
		\node (C'') at (2,0) {$c''$};
		\node (D) at (3,2) {$d$};
		\node (D') at (3,1) {$d'$};
		\node (D'') at (3,0) {$d''$};
		\node (E) at (4,2) {$e$};
		\node (E') at (4,1) {$e'$};
		\node (E'') at (4,0) {$e''$};
		\path[->,font=\scriptsize,>=angle 90]
		(A) edge node[above]{} (B)
		(A') edge node[above]{} (B')
		(A'') edge node[above]{} (B'')
		(C) edge node[above]{} (B)
		(C') edge node[above]{} (B')
		(C'') edge node[above]{} (B'')
		(C) edge node[above]{} (D)
		(C') edge node[above]{} (D')
		(C'') edge node[above]{} (D'')
		(E) edge node[above]{} (D)
		(E') edge node[above]{} (D')
		(E'') edge node[above]{} (D'')
		(A') edge node[left]{} (A)
		(A') edge node[left]{} (A'')
		(B') edge node[left]{} (B)
		(B') edge node[left]{} (B'')
		(C') edge node[left]{} (C)
		(C') edge node[left]{} (C'')	
		(D') edge node[left]{} (D)
		(D') edge node[left]{} (D'')
		(E') edge node[left]{} (E)
		(E') edge node[left]{} (E'');
		\end{tikzpicture}
	}
	\quad
	\xmapsto[]{\odot}
	\quad
	\raisebox{-0.5\height}{
		\begin{tikzpicture}
		\node (A) at (0,2) {$a$};
		\node (A') at (0,1) {$a'$};
		\node (A'') at (0,0) {$a''$};
		\node (B) at (1.5,2) {$b+_{c}d$};
		\node (B') at (1.5,1) {$b'+_{c'}d'$};
		\node (B'') at (1.5,0) {$b''+_{c''}d''$};
		\node (C) at (3,2) {$e$};
		\node (C') at (3,1) {$e'$};
		\node (C'') at (3,0) {$e''$};
		\path[->,font=\scriptsize,>=angle 90]
		(A) edge node[above]{} (B)
		(A') edge node[above]{} (B')
		(A'') edge node[above]{} (B'')
		(C) edge node[above]{} (B)
		(C') edge node[above]{} (B')
		(C'') edge node[above]{} (B'')
		(A') edge node[left]{} (A)
		(A') edge node[left]{} (A'')
		(B') edge node[left]{} (B)
		(B') edge node[left]{} (B'')
		(C') edge node[left]{} (C)
		(C') edge node[left]{} (C'');	
		\end{tikzpicture}
	}
	\]
	This respects identities. We prove that $\odot$ preserves
	composition in Lemma \ref{lem:InterchangeDblCat} below. 
	It is straightforward to check that the required equations are satisfied.  
	The associator and unitors are given by natural isomorphisms that arise from universal properties.  
\end{proof}

\begin{lem}
\label{lem:InterchangeDblCat}
	The assignment $\odot$ from 
	Lemma \ref{lem:SpanCospanDoubleCat} 
	preserves composition. 
	In particular, $\odot$ is a functor.
\end{lem}

\begin{proof}
	Let $\alpha$, $\alpha^\prime$, $\beta$, and $\beta^\prime$ 
	be composable 2-morphisms given by
	\[
	\begin{tikzpicture}
	\node () at (-0.75,1) {$\alpha =$};
	\node (A) at (0,2) {$a$};
	\node (A') at (0,1) {$a'$};
	\node (A'') at (0,0) {$\ell$};
	\node (B) at (1,2) {$b$};
	\node (B') at (1,1) {$b'$};
	\node (B'') at (1,0) {$m$};
	\node (C) at (2,2) {$c$};
	\node (C') at (2,1) {$c'$};
	\node (C'') at (2,0) {$n$};
	\path[->,font=\scriptsize,>=angle 90]
	(A) edge node[above]{$$} (B)
	(A') edge node[above]{$$} (B')
	(A'') edge node[above]{$$} (B'')
	(C) edge node[above]{$$} (B)
	(C') edge node[above]{$$} (B')
	(C'') edge node[above]{$$} (B'')
	(A') edge node[left]{$\cong$} (A)
	(A') edge node[left]{$\cong$} (A'')
	(B') edge[>->] node[left]{} (B)
	(B') edge[>->] node[left]{} (B'')
	(C') edge node[right]{$\cong$} (C)
	(C') edge node[right]{$\cong$} (C'');	
	\end{tikzpicture}
	\quad \quad
	\begin{tikzpicture}
	\node () at (-0.75,1) {$\alpha' =$};
	\node (A) at (0,2) {$c$};
	\node (A') at (0,1) {$c'$};
	\node (A'') at (0,0) {$n$};
	\node (B) at (1,2) {$d$};
	\node (B') at (1,1) {$d'$};
	\node (B'') at (1,0) {$p$};
	\node (C) at (2,2) {$e$};
	\node (C') at (2,1) {$e'$};
	\node (C'') at (2,0) {$q$};
	\path[->,font=\scriptsize,>=angle 90]
	(A) edge node[above]{$$} (B)
	(A') edge node[above]{$$} (B')
	(A'') edge node[above]{$$} (B'')
	(C) edge node[above]{$$} (B)
	(C') edge node[above]{$$} (B')
	(C'') edge node[above]{$$} (B'')
	(A') edge node[left]{$\cong$} (A)
	(A') edge node[left]{$\cong$} (A'')
	(B') edge[>->] node[left]{} (B)
	(B') edge[>->] node[left]{} (B'')
	(C') edge node[right]{$\cong$} (C)
	(C') edge node[right]{$\cong$} (C'');	
	\end{tikzpicture}
	\]
	\[
	\begin{tikzpicture}
	\node () at (-0.75,1) {$\beta =$};
	\node (A) at (0,2) {$\ell$};
	\node (A') at (0,1) {$v'$};
	\node (A'') at (0,0) {$v$};
	\node (B) at (1,2) {$m$};
	\node (B') at (1,1) {$w'$};
	\node (B'') at (1,0) {$w$};
	\node (C) at (2,2) {$n$};
	\node (C') at (2,1) {$x'$};
	\node (C'') at (2,0) {$x$};
	\path[->,font=\scriptsize,>=angle 90]
	(A) edge node[above]{$$} (B)
	(A') edge node[above]{$$} (B')
	(A'') edge node[above]{$$} (B'')
	(C) edge node[above]{$$} (B)
	(C') edge node[above]{$$} (B')
	(C'') edge node[above]{$$} (B'')
	(A') edge node[left]{$\cong$} (A)
	(A') edge node[left]{$\cong$} (A'')
	(B') edge[>->] node[left]{} (B)
	(B') edge[>->] node[left]{} (B'')
	(C') edge node[right]{$\cong$} (C)
	(C') edge node[right]{$\cong$} (C'');	
	\end{tikzpicture}
	\quad \quad
	\begin{tikzpicture}
	\node () at (-0.75,1) {$\beta' =$};
	\node (A) at (0,2) {$n$};
	\node (A') at (0,1) {$x'$};
	\node (A'') at (0,0) {$x$};
	\node (B) at (1,2) {$p$};
	\node (B') at (1,1) {$y'$};
	\node (B'') at (1,0) {$y$};
	\node (C) at (2,2) {$q$};
	\node (C') at (2,1) {$z'$};
	\node (C'') at (2,0) {$z$};
	\path[->,font=\scriptsize,>=angle 90]
	(A) edge node[above]{$$} (B)
	(A') edge node[above]{$$} (B')
	(A'') edge node[above]{$$} (B'')
	(C) edge node[above]{$$} (B)
	(C') edge node[above]{$$} (B')
	(C'') edge node[above]{$$} (B'')
	(A') edge node[left]{$\cong$} (A)
	(A') edge node[left]{$\cong$} (A'')
	(B') edge[>->] node[left]{} (B)
	(B') edge[>->] node[left]{} (B'')
	(C') edge node[right]{$\cong$} (C)
	(C') edge node[right]{$\cong$} (C'');	
	\end{tikzpicture}
	\]
	Our goal is to show that
	\begin{equation}
	\label{eq:InterchangeCspSpan}
		(\alpha \odot \alpha') \circ (\beta \odot \beta')
		=
		(\alpha \circ \beta) \odot (\alpha' \circ \beta').
	\end{equation}
	The left hand side of \eqref{eq:InterchangeCspSpan} 
	corresponds to performing horizontal composition 
	before vertical composition. 
	The right hand side of \eqref{eq:InterchangeCspSpan} reverses that order.
	
	First, we compute the left hand side 
	of \eqref{eq:InterchangeCspSpan}. 
	Composing horizontally, 
	$\alpha \odot \alpha'$ and $\beta \odot \beta'$ 
	are, respectively,
	\[
	\begin{tikzpicture}
		\node (A) at (1,1) {$a$};
		\node (A') at (1,0) {$a'$};
		\node (A'') at (1,-1) {$\ell$};
		\node (B) at (3,1) {$b+_{c}d$};
		\node (B') at (3,0) {$b'+_{c'}d'$};
		\node (B'') at (3,-1) {$m+_{n}p$};
		\node (C) at (5,1) {$e$};
		\node (C') at (5,0) {$e'$};
		\node (C'') at (5,-1) {$q$};
		\path[->,font=\scriptsize,>=angle 90]
		(A) edge node[above]{$$} (B)
		(A') edge node[above]{$$} (B')
		(A'') edge node[above]{$$} (B'')
		(C) edge node[above]{$$} (B)
		(C') edge node[above]{$$} (B')
		(C'') edge node[above]{$$} (B'')
		(A') edge node[left]{$\cong$} (A)
		(A') edge node[left]{$\cong$} (A'')
		(B') edge[>->] node[left]{} (B)
		(B') edge[>->] node[left]{} (B'')
		(C') edge node[right]{$\cong$} (C)
		(C') edge node[right]{$\cong$} (C'');	
	\end{tikzpicture}
	\quad \quad
	\begin{tikzpicture}
		\node (A) at (1,1) {$\ell$};
		\node (A') at (1,0) {$v'$};
		\node (A'') at (1,-1) {$v$};
		\node (B) at (3,1) {$m+_{n}p$};
		\node (B') at (3,0) {$w'+_{x'}y'$};
		\node (B'') at (3,-1) {$w+_{x}y$};
		\node (C) at (5,1) {$q$};
		\node (C') at (5,0) {$z'$};
		\node (C'') at (5,-1) {$z$};
		\path[->,font=\scriptsize,>=angle 90]
		(A) edge node[above]{$$} (B)
		(A') edge node[above]{$$} (B')
		(A'') edge node[above]{$$} (B'')
		(C) edge node[above]{$$} (B)
		(C') edge node[above]{$$} (B')
		(C'') edge node[above]{$$} (B'')
		(A') edge node[left]{$\cong$} (A)
		(A') edge node[left]{$\cong$} (A'')
		(B') edge[>->] node[left]{} (B)
		(B') edge[>->] node[left]{} (B'')
		(C') edge node[right]{$\cong$} (C)
		(C') edge node[right]{$\cong$} (C'');	
	\end{tikzpicture}
	\]
	The preservation of monics under 
	this operation follows from 
		\cite[Lem.~ 2.1]{Cic}.  
	Next, vertically composing 
	$\alpha \odot \alpha^\prime$ and 
	$\beta \odot \beta^\prime$, we get that 
	$(\alpha \odot \alpha^\prime) \circ (\beta \odot \beta^\prime)$ 
	is equal to
	\begin{equation}
	\label{diag:IntrchngHorVertCspSpan}
	\raisebox{-0.5\height}{
		\begin{tikzpicture}
		\node (A) at (1,1) {$a$};
		\node (A') at (1,0) {$a'\times_{\ell}v'$};
		\node (A'') at (1,-1) {$v$};
		\node (B) at (5,1) {$b+_{d}d$};
		\node (B') at (5,0) {$(b'+_{c'}d') \times_{(m+_{n}p)} (w'+_{x'}y')$};
		\node (B'') at (5,-1) {$w+_{x}y$};
		\node (C) at (9,1) {$e$};
		\node (C') at (9,0) {$e' \times_{q}z'$};
		\node (C'') at (9,-1) {$z$};
		\path[->,font=\scriptsize,>=angle 90]
		(A) edge node[above]{$$} (B)
		(A') edge node[above]{$$} (B')
		(A'') edge node[above]{$$} (B'')
		(C) edge node[above]{$$} (B)
		(C') edge node[above]{$$} (B')
		(C'') edge node[above]{$$} (B'')
		(A') edge node[left]{$\cong$} (A)
		(A') edge node[left]{$\cong$} (A'')
		(B') edge[>->] node[left]{} (B)
		(B') edge[>->] node[left]{} (B'')
		(C') edge node[right]{$\cong$} (C)
		(C') edge node[right]{$\cong$} (C'');	
		\end{tikzpicture}
	}
	\end{equation}
	
	Solving for the right hand side 
	of \eqref{eq:InterchangeCspSpan}, 
	we first obtain that 
	$\alpha \circ \beta$ and $\alpha' \circ \beta'$ 
	are, respectively,
	\[
	\begin{tikzpicture}
		\node (A) at (1,1) {$a$};
		\node (A') at (1,0) {$a' \times_{\ell}v'$};
		\node (A'') at (1,-1) {$v$};
		\node (B) at (3,1) {$b$};
		\node (B') at (3,0) {$b' \times_{m}w'$};
		\node (B'') at (3,-1) {$w$};
		\node (C) at (5,1) {$c$};
		\node (C') at (5,0) {$c' \times_{n}x'$};
		\node (C'') at (5,-1) {$x$};
		\path[->,font=\scriptsize,>=angle 90]
		(A) edge node[above]{$$} (B)
		(A') edge node[above]{$$} (B')
		(A'') edge node[above]{$$} (B'')
		(C) edge node[above]{$$} (B)
		(C') edge node[above]{$$} (B')
		(C'') edge node[above]{$$} (B'')
		(A') edge node[left]{$\cong$} (A)
		(A') edge node[left]{$\cong$} (A'')
		(B') edge[>->] node[left]{} (B)
		(B') edge[>->] node[left]{} (B'')
		(C') edge node[right]{$\cong$} (C)
		(C') edge node[right]{$\cong$} (C'');	
	\end{tikzpicture}
	\quad \quad 
	\begin{tikzpicture}
		\node (A) at (1,1) {$c$};
		\node (A') at (1,0) {$c' \times_{n}x'$};
		\node (A'') at (1,-1) {$x$};
		\node (B) at (3,1) {$d$};
		\node (B') at (3,0) {$d' \times_{p}y'$};
		\node (B'') at (3,-1) {$y$};
		\node (C) at (5,1) {$e$};
		\node (C') at (5,0) {$e' \times_{q}z'$};
		\node (C'') at (5,-1) {$z$};
		\path[->,font=\scriptsize,>=angle 90]
		(A) edge node[above]{$$} (B)
		(A') edge node[above]{$$} (B')
		(A'') edge node[above]{$$} (B'')
		(C) edge node[above]{$$} (B)
		(C') edge node[above]{$$} (B')
		(C'') edge node[above]{$$} (B'')
		(A') edge node[left]{$\cong$} (A)
		(A') edge node[left]{$\cong$} (A'')
		(B') edge[>->] node[left]{} (B)
		(B') edge[>->] node[left]{} (B'')
		(C') edge node[right]{$\cong$} (C)
		(C') edge node[right]{$\cong$} (C'');	
	\end{tikzpicture}
	\]
	Composing these horizontally, 
	we get that 
	$(\alpha \circ \beta) \odot (\alpha^\prime \circ \beta^\prime)$ 
	equals
	\begin{equation}
	\label{diag:IntrchngVertHorCspSpan}
	\raisebox{-0.5\height}{
		\begin{tikzpicture}
		\node (A) at (1,1) {$a$};
		\node (A') at (1,0) {$a' \times_{\ell}v'$};
		\node (A'') at (1,-1) {$v$};
		\node (B) at (5,1) {$b +_{c}d$};
		\node (B') at (5,0) {$(b'\times_{m}w')+_{(c'\times_{n}x')}(d'\times_{p}y')$};
		\node (B'') at (5,-1) {$w+_{x}y$};
		\node (C) at (9,1) {$e$};
		\node (C') at (9,0) {$e' \times_{q}z'$};
		\node (C'') at (9,-1) {$z$};
		\path[->,font=\scriptsize,>=angle 90]
		(A) edge node[above]{$$} (B)
		(A') edge node[above]{$$} (B')
		(A'') edge node[above]{$$} (B'')
		(C) edge node[above]{$$} (B)
		(C') edge node[above]{$$} (B')
		(C'') edge node[above]{$$} (B'')
		(A') edge node[left]{$\cong$} (A)
		(A') edge node[left]{$\cong$} (A'')
		(B') edge[>->] node[left]{} (B)
		(B') edge[>->] node[left]{} (B'')
		(C') edge node[right]{$\cong$} (C)
		(C') edge node[right]{$\cong$} (C'');	
		\end{tikzpicture}
	}
	\end{equation}
	
	Now, we need to show that 
	\eqref{diag:IntrchngHorVertCspSpan} is equal to 
	\eqref{diag:IntrchngVertHorCspSpan} 
	as 2-morphisms.  
	Note that the diagrams only differ in the middle.  
	Thus, to complete the interchange law, 
	it suffices to establish an isomorphism 
	\[
		(b'+_{c'}d') \times_{(m+_{n}p)} (w'+_{x'}y')
		\to 
		(b' \times_{m}w')+_{(c' \times_{n}x')}(d' \times_{p}y')
	\]
	But because the left and right vertical spans 
	have isomorphisms for legs, 
	the isomorphism we seek follows from \cite[Lem.~2.5]{Cic}. 
\end{proof}

\begin{lem}
\label{lem:SpanCospanSM}
	$\dblmonspcsp{T}$ is a symmetric monoidal double category.  
\end{lem}

\begin{proof}
	Let us first show that the category of objects 
		$\dblcat{M}_0$ 
	and the category of arrows 
		$\dblcat{M}_1$ 
	are symmetric monoidal categories.  
	Note that $\dblcat{M}_0$ is the 
	largest groupoid contained in $\mathbf{Sp(T)}$. 
	We obtain the monoidal structure on $\dblcat{M}_{0}$ by
	lifting the cocartesian structure on $\cat{T}$
	to the objects and by defining
	\[
		 (b \from a \tospan c) + (b' \from a' \tospan c')
		 =
		 (b+b' \from a+a' \tospan c+c')
	\]
	on morphisms.  
	Universal properties provide the 
	associator and unitors as well as 
	the coherence axioms. 
	This monoidal structure is clearly symmetric.
	
	Next, we have that $\dblcat{M}_1$ is the category whose 
	objects are the cospans in $\cat{T}$ and 
	morphisms are the isomorphism classes of monic spans of cospans in $\cat{T}$.
	We obtain a symmetric monoidal structure on the objects via 
	\[
	(b \from a \tocospan c) + (b' \from a' \tocospan c')
	=
	(b+b' \from a+a' \tocospan c+c')
	\]
	and on the morphisms by
	\[
	\raisebox{-0.5\height}{
		\begin{tikzpicture}
		\node (A) at (0,2) {$\bullet$};
		\node (A') at (0,1) {$\bullet$};
		\node (A'') at (0,0) {$\bullet$};
		\node (B) at (1,2) {$\bullet$};
		\node (B') at (1,1) {$\bullet$};
		\node (B'') at (1,0) {$\bullet$};
		\node (C) at (2,2) {$\bullet$};
		\node (C') at (2,1) {$\bullet$};
		\node (C'') at (2,0) {$\bullet$};
		\path[->,font=\scriptsize,>=angle 90]
		(A) edge node[above]{} (B)
		(A') edge node[above]{} (B')
		(A'') edge node[above]{} (B'')
		(C) edge node[above]{} (B)
		(C') edge node[above]{} (B')
		(C'') edge node[above]{} (B'')
		(A') edge node[left]{} (A)
		(A') edge node[left]{} (A'')
		(B') edge[>->] node[left]{} (B)
		(B') edge[>->] node[left]{} (B'')
		(C') edge node[left]{} (C)
		(C') edge node[left]{} (C'');	
		\end{tikzpicture}
	}
	\quad + \quad
	\raisebox{-0.5\height}{
		\begin{tikzpicture}
		\node (A) at (0,2) {$\ast$};
		\node (A') at (0,1) {$\ast$};
		\node (A'') at (0,0) {$\ast$};
		\node (B) at (1,2) {$\ast$};
		\node (B') at (1,1) {$\ast$};
		\node (B'') at (1,0) {$\ast$};
		\node (C) at (2,2) {$\ast$};
		\node (C') at (2,1) {$\ast$};
		\node (C'') at (2,0) {$\ast$};
		\path[->,font=\scriptsize,>=angle 90]
		(A) edge node[above]{} (B)
		(A') edge node[above]{} (B')
		(A'') edge node[above]{} (B'')
		(C) edge node[above]{} (B)
		(C') edge node[above]{} (B')
		(C'') edge node[above]{} (B'')
		(A') edge node[left]{} (A)
		(A') edge node[left]{} (A'')
		(B') edge[>->] node[left]{} (B)
		(B') edge[>->] node[left]{} (B'')
		(C') edge node[left]{} (C)
		(C') edge node[left]{} (C'');	
		\end{tikzpicture}
	}
	\quad = \quad
	\raisebox{-0.5\height}{
		\begin{tikzpicture}
		\node (A) at (0,2) {$\bullet+\ast$};
		\node (A') at (0,1) {$\bullet+\ast$};
		\node (A'') at (0,0) {$\bullet+\ast$};
		\node (B) at (1.5,2) {$\bullet+\ast$};
		\node (B') at (1.5,1) {$\bullet+\ast$};
		\node (B'') at (1.5,0) {$\bullet+\ast$};
		\node (C) at (3,2) {$\bullet+\ast$};
		\node (C') at (3,1) {$\bullet+\ast$};
		\node (C'') at (3,0) {$\bullet+\ast$};
		\path[->,font=\scriptsize,>=angle 90]
		(A) edge node[above]{} (B)
		(A') edge node[above]{} (B')
		(A'') edge node[above]{} (B'')
		(C) edge node[above]{} (B)
		(C') edge node[above]{} (B')
		(C'') edge node[above]{} (B'')
		(A') edge node[left]{} (A)
		(A') edge node[left]{} (A'')
		(B') edge[>->] node[left]{} (B)
		(B') edge[>->] node[left]{} (B'')
		(C') edge node[left]{} (C)
		(C') edge node[left]{} (C'');	
		\end{tikzpicture}
	}
	\]
	Again, universal properties provide 
	the associator, unitors, and coherence axioms.  
	Hence both $\dblcat{M}_0$ and $\dblcat{M}_1$ 
	are symmetric monoidal categories.

It remains to find globular isomorphisms 
	$\mathfrak{x}$ 
	and $\mathfrak{u}$ such that 
	the required diagrams commute. 
	To find $\mathfrak{x}$, fix horizontal 1-morphisms 
	\begin{align*}
		b & \from a \tocospan c, & w &\from v \tocospan x, \\
		d & \from c \tocospan e, & y &\from x \tocospan z'.
	\end{align*}
	The globular isomorphism $\mathfrak{x}$ is an invertible 2-morphism with domain
	\[
		(a+v) \to (b+w) +_{(c+x)} (d+y) \gets (e+z)
	\]
	and codomain
	\[
		(a+v) \to (b+_c d) + (w+_x y) \gets (e+z).
	\]
	This comes down to finding an isomorphism in $\mathbf{T}$ 
	between the apices of the above cospans.  
	Such an isomorphism exists, and is unique, 
	because both apices are colimits of the non-connected diagram
	\[
		\begin{tikzpicture}
			\node (a) at (0,0) {$a$};
			\node (b) at (1,.5) {$b$};
			\node (c) at (2,0) {$c$};
			\node (d) at (3,.5) {$d$};
			\node (e) at (4,0) {$e$};
			\node (v) at (5,0) {$v$};
			\node (w) at (6,.5) {$w$};
			\node (x) at (7,0) {$x$};
			\node (y) at (8,.5) {$y$};
			\node (z) at (9,0) {$z$};
			\path[->,font=\scriptsize,>=angle 90]
			(a) edge node[above]{$$} (b)
			(c) edge node[above]{$$} (b)
			(c) edge node[above]{$$} (d)
			(e) edge node[above]{$$} (d)
			(v) edge node[above]{$$} (w)
			(x) edge node[above]{$$} (w)
			(x) edge node[above]{$$} (y)
			(z) edge node[above]{$$} (y);
		\end{tikzpicture}
	\]
	Moreover, the resulting globular isomorphism is a monic span of cospans as the universal maps are isomorphisms. The globular isomorphism $\mathfrak{u}$ is similar. 
	
	Finally, we check that the coherence axioms, 
	namely (a)-(k) of Definition 
		\ref{def:MonoidalDoubleCategory}, 
	hold. 
	These are straightforward, though tedious, to verify. 
	For instance, if we have
	\[
		\begin{tikzpicture}
			\node (a) at (0,0) {$a$};
			\node (b) at (1,.5) {$b$};
			\node (c) at (2,0) {$c$};
			\node (M1) at (-1,.25) {$M_1 =$};
			\node (M2) at (3,.25) {$M_2 =$};
			\node (c2) at (4,0) {$c$};
			\node (d) at (5,.5) {$d$};
			\node (e) at (6,0) {$e$};
			\node (M3) at (7,.25) {$M_3 =$};
			\node (e2) at (8,0) {$e$};
			\node (f) at (9,.5) {$f$};
			\node (g) at (10,0) {$g$};
			\node (t) at (0,-2) {$t$};
			\node (u) at (1,-1.5) {$u$};
			\node (v) at (2,-2) {$v$};
			\node (N1) at (-1,-1.75) {$N_1 =$};
			\node (N2) at (3,-1.75) {$N_2 =$};
			\node (v2) at (4,-2) {$v$};
			\node (w) at (5,-1.5) {$w$};
			\node (x) at (6,-2) {$x$};
			\node (N3) at (7,-1.75) {$N_3 =$};
			\node (x2) at (8,-2) {$x$};
			\node (y) at (9,-1.5) {$y$};
			\node (z) at (10,-2) {$z$};
			\path[->,font=\scriptsize,>=angle 90]
			(a) edge node[above]{$$} (b)
			(c) edge node[above]{$$} (b)
			(c2) edge node[above]{$$} (d)
			(e) edge node[above]{$$} (d)
			(v2) edge node[above]{$$} (w)
			(e2) edge node[above]{$$} (f)
			(g) edge node[above]{$$} (f)
			(t) edge node[above]{$$} (u)
			(v) edge node[above]{$$} (u)
			(x) edge node[above]{$$} (w)
			(x2) edge node[above]{$$} (y)
			(z) edge node[above]{$$} (y);
		\end{tikzpicture}
	\]
then following diagram \eqref{diag:MonDblCat} 
around the top right gives the sequence of cospans
\[
		\begin{tikzpicture}
			\node (a) at (-3,0) {$a+t$};
			\node (b) at (1,.5) {$((b+w)+_{c+v}(d+w))+_{e+x}(f+y)$};
			\node (c) at (5,0) {$g+z$};
			\node (M1) at (-2,1.5) {$((M_1 \otimes N_1) \odot (M_2 \otimes N_2)) \odot (M_3 \otimes N_3) =$};
			\path[->,font=\scriptsize]
			(a) edge[in=180,out=45] node[above]{$$} (b)
			(c) edge[in=0,out=135] node[above]{$$} (b);
		\end{tikzpicture}
	\]
	\[
		\begin{tikzpicture}
			\node (a) at (-3,0) {$a+t$};
			\node (b) at (1,.5) {$((b+_c d)+ (u+_v w))+_{e+x}(f+y)$};
			\node (c) at (5,0) {$g+z$};
			\node (M1) at (-2,1.5) {$((M_1 \odot M_2) \otimes (N_1 \odot N_2)) \odot (M_3 \otimes N_3)=$};
			\path[->,font=\scriptsize]
			(a) edge[in=180,out=45] node[above]{$$} (b)
			(c) edge[in=0,out=135] node[above]{$$} (b);
		\end{tikzpicture}
	\]
	\[
		\begin{tikzpicture}
			\node (a) at (-3,0) {$a+t$};
			\node (b) at (1,.5) {$((b+_c d)+_e f)+ ((u+_v w) +_x y)$};
			\node (c) at (5,0) {$g+z$};
			\node (M1) at (-2,1.5) {$((M_1 \odot M_2) \odot M_3) \otimes ((N_1 \odot N_2) \odot N_3)=$};
			\path[->,font=\scriptsize]
			(a) edge[in=180,out=45] node[above]{$$} (b)
			(c) edge[in=0,out=135] node[above]{$$} (b);
		\end{tikzpicture}
	\]
	\[
		\begin{tikzpicture}
			\node (a) at (-3,0) {$a+t$};
			\node (b) at (1,.5) {$(b+_c (d+_e f)) + (u+_v (w+_x y))$};
			\node (c) at (5,0) {$g+z$};
			\node (M1) at (-2,1.5) {$(M_1 \odot (M_2 \odot M_3)) \otimes (N_1 \odot (N_2 \odot N_3)) =$};
			\path[->,font=\scriptsize]
			(a) edge[in=180,out=45] node[above]{$$} (b)
			(c) edge[in=0,out=135] node[above]{$$} (b);
		\end{tikzpicture}
	\]
Following the diagram \eqref{diag:MonDblCat} 
around the bottom left gives another sequence of cospans
\[
		\begin{tikzpicture}
		\node (a) at (-3,0) {$a+t$};
		\node (b) at (1,.5) {$((b+w)+_{c+v}(d+w))+_{e+x}(f+y)$};
		\node (c) at (5,0) {$g+z$};
		\node (M1) at (-2,1.5) {$((M_1 \otimes N_1) \odot (M_2 \otimes N_2)) \odot (M_3 \otimes N_3) =$};
		\path[->,font=\scriptsize]
		(a) edge[in=180,out=45] node[above]{$$} (b)
		(c) edge[in=0,out=135] node[above]{$$} (b);
		\end{tikzpicture}
	\]
\[
		\begin{tikzpicture}
			\node (a) at (-3,0) {$a+t$};
			\node (b) at (1,.5) {$(b+w)+_{c+v}((d+w)+_{e+x}(f+y))$};
			\node (c) at (5,0) {$g+z$};
			\node (M1) at (-2,1.5) {$(M_1 \otimes N_1) \odot ((M_2 \otimes N_2) \odot (M_3 \otimes N_3))=$};
			\path[->,font=\scriptsize]
			(a) edge[in=180,out=45] node[above]{$$} (b)
			(c) edge[in=0,out=135] node[above]{$$} (b);
		\end{tikzpicture}
	\]
\[
		\begin{tikzpicture}
			\node (a) at (-3,0) {$a+t$};
			\node (b) at (1,.5) {$(b+u)+_{c+v}((d+_e f) + (w+_x y))$};
			\node (c) at (5,0) {$g+z$};
			\node (M1) at (-2,1.5) {$(M_1 \otimes N_1) \odot ((M_2 \odot M_3) \otimes (N_2 \odot N_3))=$};
			\path[->,font=\scriptsize]
			(a) edge[in=180,out=45] node[above]{$$} (b)
			(c) edge[in=0,out=135] node[above]{$$} (b);
		\end{tikzpicture}
	\]
\[
		\begin{tikzpicture}
			\node (a) at (-3,0) {$a+t$};
			\node (b) at (1,.5) {$(b+_c (d+_e f)) + (u+_v (w+_x y))$};
			\node (c) at (5,0) {$g+z$};
			\node (M1) at (-2,1.5) {$(M_1 \odot (M_2 \odot M_3)) \otimes (N_1 \odot (N_2 \odot N_3)) =$};
			\path[->,font=\scriptsize]
			(a) edge[in=180,out=45] node[above]{$$} (b)
			(c) edge[in=0,out=135] node[above]{$$} (b);
		\end{tikzpicture}
	\]
Putting these together gives the following commutative diagram.
\[
		\begin{tikzpicture}
			\node (a) at (-4,0) {$a+t$};
			\node (b) at (1,0) {$((b+w)+_{c+v}(d+w))+_{e+x}(f+y)$};
			\node (c) at (6,0) {$g+z$};
			\node (a2) at (-4,1.5) {$a+t$};
			\node (b2) at (1,1.5) {$((b+_c d)+ (u+_v w))+_{e+x}(f+y)$};
			\node (c2) at (6,1.5) {$g+z$};
                                \node (a3) at (-4,3) {$a+t$};
			\node (b3) at (1,3) {$((b+_c d)+_e f)+ ((u+_v w) +_x y)$};
			\node (c3) at (6,3) {$g+z$};
                                \node (a4) at (-4,4.5) {$a+t$};
			\node (b4) at (1,4.5) {$(b+_c (d+_e f)) + (u+_v (w+_x y))$};
			\node (c4) at (6,4.5) {$g+z$};
                                \node (a5) at (-4,-1.5) {$a+t$};
			\node (b5) at (1,-1.5) {$(b+w)+_{c+v}((d+w)+_{e+x}(f+y))$};
			\node (c5) at (6,-1.5) {$g+z$};
                                \node (a6) at (-4,-3) {$a+t$};
			\node (b6) at (1,-3) {$(b+u)+_{c+v}((d+_e f) + (w+_x y))$};
			\node (c6) at (6,-3) {$g+z$};
                                \node (a7) at (-4,-4.5) {$a+t$};
			\node (b7) at (1,-4.5) {$(b+_c (d+_e f)) + (u+_v (w+_x y))$};
			\node (c7) at (6,-4.5) {$g+z$};
			\path[->,font=\scriptsize,>=angle 90]
			(a) edge node[above]{$$} (b)
			(c) edge node[above]{$$} (b)
                                (a2) edge node[above]{$$} (b2)
			(c2) edge node[above]{$$} (b2)
                                (a) edge node[above]{$$} (a2)
                                (b) edge node[above]{$$} (b2)
			(c) edge node[above]{$$} (c2)
                                (a3) edge node[above]{$$} (b3)
			(c3) edge node[above]{$$} (b3)
                                (a2) edge node[above]{$$} (a3)
                                (b2) edge node[above]{$$} (b3)
			(c2) edge node[above]{$$} (c3)
                                (a4) edge node[above]{$$} (b4)
			(c4) edge node[above]{$$} (b4)
                                (a3) edge node[above]{$$} (a4)
                                (b3) edge node[above]{$$} (b4)
			(c3) edge node[above]{$$} (c4)
                                (a5) edge node[above]{$$} (b5)
			(c5) edge node[above]{$$} (b5)
                                (a) edge node[above]{$$} (a5)
                                (b) edge node[above]{$$} (b5)
			(c) edge node[above]{$$} (c5)
                                (a6) edge node[above]{$$} (b6)
			(c6) edge node[above]{$$} (b6)
                                (a5) edge node[above]{$$} (a6)
                                (b5) edge node[above]{$$} (b6)
			(c5) edge node[above]{$$} (c6)
                                (a7) edge node[above]{$$} (b7)
			(c7) edge node[above]{$$} (b7)
                                (a6) edge node[above]{$$} (a7)
                                (b6) edge node[above]{$$} (b7)
			(c6) edge node[above]{$$} (c7);
		\end{tikzpicture}
	\]
The vertical 1-morphisms on the left and right 
are the the respective
identity spans on $a+t$ and $g+z$. 
The vertical 1-morphisms in the center 
are isomorphism classes of monic spans 
where each leg is given by 
a universal map between two colimits 
of the same diagram. 
The horizontal 1-morphisms are 
given by universal maps into coproducts and pushouts.
The top cospan is the same as the bottom cospan, 
making a bracelet-like figure 
in which all faces commute.  
The other diagrams witnessing coherence 
are given in a similar fashion.
\end{proof}

\begin{lem}
\label{lem:SpanCospanIsofibrant}
	The symmetric monoidal double category $\dblmonspcsp{T}$ is isofibrant.  
\end{lem}

\begin{proof}
	The companion of a vertical 1-morphism 
		$f = (b \from a \tospan c)$ 
	is given by
		$\widehat{f} = (b \from a \tocospan c)$ 
	whose legs are the inverses of the legs of $f$. The required 2-morphisms are given by
	\[
	\raisebox{-0.5\height}{
	\begin{tikzpicture}
		\node (A) at (0,2) {$a$};
		\node (A') at (0,1) {$b$};
		\node (A'') at (0,0) {$c$};
		\node (B) at (1,2) {$b$};
		\node (B') at (1,1) {$c$};
		\node (B'') at (1,0) {$c$};
		\node (C) at (2,2) {$c$};
		\node (C') at (2,1) {$c$};
		\node (C'') at (2,0) {$c$};
		\path[->,font=\scriptsize,>=angle 90]
		(A) edge node[above]{} (B)
		(A') edge node[above]{} (B')
		(A'') edge node[above]{} (B'')
		(C) edge node[above]{} (B)
		(C') edge node[above]{} (B')
		(C'') edge node[above]{} (B'')
		(A') edge node[left]{} (A)
		(A') edge node[left]{} (A'')
		(B') edge node[left]{} (B)
		(B') edge node[left]{} (B'')
		(C') edge node[left]{} (C)
		(C') edge node[left]{} (C'');
	\end{tikzpicture}
	}
	\t{ and }
	\raisebox{-0.5\height}{
	\begin{tikzpicture}
		\node (A) at (0,2) {$a$};
		\node (A') at (0,1) {$a$};
		\node (A'') at (0,0) {$a$};
		\node (B) at (1,2) {$a$};
		\node (B') at (1,1) {$a$};
		\node (B'') at (1,0) {$b$};
		\node (C) at (2,2) {$a$};
		\node (C') at (2,1) {$b$};
		\node (C'') at (2,0) {$c$};
		\path[->,font=\scriptsize,>=angle 90]
		(A) edge node[above]{} (B)
		(A') edge node[above]{} (B')
		(A'') edge node[above]{} (B'')
		(C) edge node[above]{} (B)
		(C') edge node[above]{} (B')
		(C'') edge node[above]{} (B'')
		(A') edge node[left]{} (A)
		(A') edge node[left]{} (A'')
		(B') edge node[left]{} (B)
		(B') edge node[left]{} (B'')
		(C') edge node[left]{} (C)
		(C') edge node[left]{} (C'');
	\end{tikzpicture}
	}
	\]
	The conjoint of $f$ is given by $\check{f} = \widehat{f}^{\text{op}}$.
\end{proof}


The benefit of laying down this groundwork
is that the following theorem
now follows from applying
Theorem \ref{thm:DoubleGivesBi}
to $\bimonspcsp{T}$.

\begin{thm}
	\label{thm:SpansCospasAreSMBicat}
	$\bimonspcsp{T}$ is a symmetric monoidal bicategory.
\end{thm}

\begin{proof}
We have shown that $\mathbf{\mathbb{M} onicSp(Csp(T))}$ is an isofibrant symmetric monoidal pseudo double category and so we obtain $\mathbf{MonicSp(Csp(T))}$ as the horizontal edge bicategory of $\mathbf{\mathbb{M} onicSp(Csp(T))}$ by Shulman's result.
\end{proof}

It remains to show that this bicategory
is compact closed. 
We start with the following lemma.

\begin{lem}
\label{lem:PushoutDiagram}
	The diagram
	\[
		\begin{tikzpicture}
			\node (UL) at (0,1) {$X+X+X$};
			\node (LL) at (0,0) {$X+X$};
			\node (UR) at (3,1) {$X+X$};
			\node (LR) at (3,0) {$X$};
			\path[->,font=\scriptsize,>=angle 90]
			(UL) edge node[above] {$X+\nabla$} (UR)
			(UL) edge node[left] {$\nabla +X$} (LL)
			(UR) edge node[right] {$\nabla$} (LR)
			(LL) edge node[above] {$\nabla$} (LR);
		\end{tikzpicture}
	\]
	is a pushout square.
\end{lem}

\begin{proof}
	Suppose that we have maps 
		$f,g \from X+X \to Y$
	forming a cocone over the span 
	inside the above diagram. 
	Let $\iota \from X \to X+X+X$ include $X$ 
	into the middle copy. 
	Observe that 
		$\ell \coloneqq (\nabla + X) \circ \iota$ 
	and 
		$r \coloneqq (X + \nabla) \circ \iota$ 
	are, respectively, the left and right inclusions 
		$X \to X+X$. 
	Then $f \circ \ell = g \circ r$ is a map $X \to Y$, 
	which we claim is the unique map making 
	the required diagram commute. 
	Indeed, given $h \from X \to Y$ such that 
		$f = h \circ \nabla = g$, 
	then $g \circ r = f \circ \ell = h \circ \nabla \circ \ell = h$.
\end{proof}

\begin{thm}
	\label{thm:SpansMapsAreCCBicat}
	The symmetric monoidal bicategory $\bimonspcsp{T}$ is compact closed.
\end{thm}

\begin{proof}
	First we show that each object is its own dual. 
	For an object $X$,  
	define the counit 
		$e \from X + X \tocospan 0$ 
	and unit 
		$c \from 0 \tocospan X+X$ 
	to be the following cospans:
	\[
		e = (X+X \xto{\nabla} X \gets 0), 
		\quad \quad 
		c = (0 \to X \xleftarrow{\nabla} X+X).
	\]
	Next we define the cusp isomorphisms, 
		$\alpha$ and $\beta$.
	Note that $\alpha$ is a 2-morphism 
	whose domain is the composite 
	\[
		X \xto{\ell}
		X+X \xleftarrow{X+\nabla}
		X+X+X \xto{\nabla +X}
		X+X \xleftarrow{r}
		X
	\]
	and whose codomain is the identity cospan on $X$.  
	From Lemma \ref{lem:PushoutDiagram} 
	we have the equations
		$\nabla+X = \ell \circ \nabla$ 
	and 
		$X + \nabla = r \circ \nabla$ 
	from which it follows that the domain of $\alpha$ is 
	the identity cospan on $X$, and  
	the codomain of $\beta$ is also 
	the identity cospan on $X$
	obtained as the composite 
	\[
		X \xto{r}
		X+X \xleftarrow{\nabla+X}
		X+X+X \xto{X+\nabla}
		X+X \xleftarrow{\ell}
		X
	\]
	Take $\alpha$ and $\beta$ each 
	to be the isomorphism class determined by the identity 2-morphism on $X$, which in particular is a monic span of cospans.
	Thus we have a dual pair 
		$(X,X,e,c,\alpha,\beta)$. 
	By Theorem 
		\ref{thm:StrictingDualPairs}, 
	there exists a cusp isomorphism $\beta'$ 
	such that 
		$(X,X,e,c,\alpha,\beta')$ 
	is a coherent dual pair, and thus $\mathbf{MonicSp(Csp(T))}$ is compact closed.  
\end{proof}

\section{An application} 
\label{sec:Applications}

The primary motivation for constructing $\bimonspcsp{T}$
is to provide a formalism in which to study networks
that have inputs and outputs.
Because of the symmetric monoidal and compact structure,
we know that networks can be placed side by side 
by `tensoring' and we can formally swap the inputs
and outputs by compactness. 
For instance, the first author applied this formalism
to the zx-calculus \cite{Cic_zx}, which is a graphical
language used for expressing operations
on a pair of qubits.

In fact, this construction is a natural 
fit for other graphical calculi, too. 
The following example shows how to use 
our construction to rewrite a network. 
In particular, we replace
a node with a more complex network
and see how compactness affects
this situation.  

\subsection{Replacing a node in a network}
\label{subsec:Replacing a node in a network}

The following example was inspired 
by a comment of Michael Shulman \cite{Shul} 
on the n-Cafe.
The example involves replacing a particular node 
in some given network with another network, 
possibly more complex, 
whose inputs and outputs 
coincide with those of the node. 
For simplicity, we work with open graphs.
First, we should place ourselves
in the context of the 
compact closed sub-bicategory
$\cat{Rewrite}$ of $\bimonspcsp{Graph}$ where $\mathbf{Rewrite}$ is the sub-bicategory
that is 1-full and 2-full on the edgeless graphs as objects.
The idea is that a 1-morphism is 
an \textbf{open graph}.
That is, a graph with 
\emph{input} and \emph{output} nodes 
chosen by the legs of the cospan. 
The 2-morphisms contain 
all possible ways for an open graph 
to be rewritten into another open graph
while preserving the inputs and outputs.

However, $\cat{Rewrite}$ is really only 
interesting as an ambient bicategory.  
We are particularly interested in 
sub-bicategories freely generated by 
various collections of 
1-morphisms and 2-morphisms. 
The collection considered
depends on our interests.
The generators of the sub-bicategory $\mathbf{Rewrite}$
are chosen to provide a syntax
for whatever types of networks we wish to study.

Suppose that we are working 
within a network where
whenever we see 
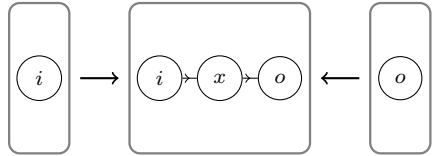
\[
\begin{tikzpicture}
%
%
\begin{scope}[shift={(0.1,0)}]
\draw[rounded corners,line width=0.3mm, gray] (-2.8,0) rectangle (-2,2);
\node [style=circle,draw=black] at (-2.4,1) {\scriptsize $i$};
\end{scope}
%
%
\begin{scope}[shift={(0.9,-3)}]
\draw[rounded corners,line width=0.3mm, gray] (-2,3) rectangle (0.4,5);
\node [style=circle,draw=black] (1) at (-0.8,4) {\scriptsize $x$};
\node [style=circle,draw=black] (3) at (0,4) {\scriptsize $o$};
\node [style=circle,draw=black] (4) at (-1.6,4) {\scriptsize $i$};
\draw [->-] (1) to (3);
\draw [->-] (4) to (1);
\end{scope}
%
%
\begin{scope}[shift={(0.1,0)}]
\draw[rounded corners,line width=0.3mm, gray] (2,0) rectangle (2.8,2);
\node [style=circle,draw=black] at (2.4,1) {\scriptsize $o$};
\end{scope}
%
%
\node (v1) at (-1.9,1) {};
\node (v2) at (-1.1,1) {};
\draw [thick,->]  (v1) edge (v2);
\node (v3) at (2.1,1) {};
\node (v4) at (1.3,1) {};
\draw [thick,->] (v3) edge (v4);
\end{tikzpicture}
\]
we want to replace it with the following:
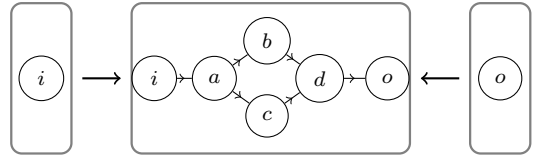
\[
\begin{tikzpicture}
%
%
\begin{scope}[shift={(-0.3,0)}]
\draw[rounded corners,line width=0.3mm, gray] (-2.8,0) rectangle (-2,2);
\node [style=circle,draw=black] at (-2.4,1) {\scriptsize $i$};
\end{scope}
%
%
\begin{scope}[shift={(0.9,-3)}]
\draw[rounded corners,line width=0.3mm, gray] (-2.4,3) rectangle (1.3,5);
\node [style=circle,draw=black] (4) at (-2.1,4) {\scriptsize $i$};
\node [style=circle,draw=black] (3) at (1,4) {\scriptsize $o$};
\node [style=circle,draw=black] (v5) at (-1.3,4) {\scriptsize $a$};
\node [style=circle,draw=black] (v6) at (-0.6,4.5) {\scriptsize $b$};
\node [style=circle,draw=black] (v7) at (-0.6,3.5) {\scriptsize $c$};
\node [style=circle,draw=black] (v8) at (0.1,4) {\scriptsize $d$};
\draw [->-]  (4) to (v5);
\draw [->-] (v5) to (v6);
\draw [->-] (v5) to (v7);
\draw [->-] (v6) to (v8);
\draw [->-] (v7) to (v8);
\draw [->-] (v8) to (3);
\end{scope}
%
%
\begin{scope}[shift={(1,0)}]
\draw[rounded corners,line width=0.3mm, gray] (2,0) rectangle (2.8,2);
\node [style=circle,draw=black] at (2.4,1) {\scriptsize $o$};
\end{scope}
%
%
\node (v1) at (-2.3,1) {};
\node (v2) at (-1.5,1) {};
\draw [thick,->]  (v1) edge (v2);
\node (v3) at (3,1) {};
\node (v4) at (2.2,1) {};
\draw [thick,->] (v3) edge (v4);
\end{tikzpicture}
\]
This corresponds to having
the $2$-morphism
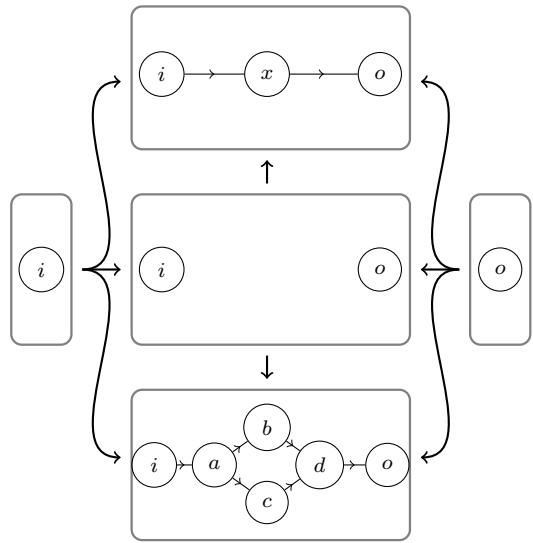
\[ 
\begin{tikzpicture}
%
%
\begin{scope}[shift={(-0.3,0)}]
\draw[rounded corners,line width=0.3mm, gray] (-2.8,0) rectangle (-2,2);
\node [style=circle,draw=black] at (-2.4,1) {\scriptsize $i$};
\end{scope}
%
%
\begin{scope}[shift={(0.5,-3)}]
\draw[rounded corners,line width=0.3mm, gray] (-2,3) rectangle (1.7,5);
\node [style=circle,draw=black] (3) at (1.3,4) {\scriptsize $o$};
\node [style=circle,draw=black] (4) at (-1.6,4) {\scriptsize $i$};
\end{scope}
%
%
\begin{scope}[shift={(0.8,-0.4)}]
\draw[rounded corners,line width=0.3mm, gray] (-2.3,3) rectangle (1.4,4.9);
\node [style=circle,draw=black] (1) at (-0.5,4) {\scriptsize $x$};
\node [style=circle,draw=black] (3) at (1,4) {\scriptsize $o$};
\node [style=circle,draw=black] (4) at (-1.9,4) {\scriptsize $i$};
\draw [->-] (1) to (3);
\draw [->-] (4) to (1);
\end{scope}
%
%
\begin{scope}[shift={(0.9,-5.6)}]
\draw[rounded corners,line width=0.3mm, gray] (-2.4,3) rectangle (1.3,5);
\node [style=circle,draw=black] (4) at (-2.1,4) {\scriptsize $i$};
\node [style=circle,draw=black] (3) at (1,4) {\scriptsize $o$};
\node [style=circle,draw=black] (v5) at (-1.3,4) {\scriptsize $a$};
\node [style=circle,draw=black] (v6) at (-0.6,4.5) {\scriptsize $b$};
\node [style=circle,draw=black] (v7) at (-0.6,3.5) {\scriptsize $c$};
\node [style=circle,draw=black] (v8) at (0.1,4) {\scriptsize $d$};
\draw [->-]  (4) to (v5);
\draw [->-] (v5) to (v6);
\draw [->-] (v5) to (v7);
\draw [->-] (v6) to (v8);
\draw [->-] (v7) to (v8);
\draw [->-] (v8) to (3);
\end{scope}
%
%
\begin{scope}[shift={(1,0)}]
\draw[rounded corners,line width=0.3mm, gray] (2,0) rectangle (2.8,2);
\node [style=circle,draw=black] at (2.4,1) {\scriptsize $o$};
\end{scope}
%
%
\node (v1) at (-2.3,1) {};
\node (v2) at (-1.5,1) {};
\node (v3) at (3,1) {};
\node (v4) at (2.2,1) {};
\node (v9) at (-1.5,3.5) {};
\node (v10) at (-1.5,-1.5) {};
\node (v11) at (2.2,3.5) {};
\node (v12) at (2.2,-1.5) {};
\node (v14) at (0.3,2.6) {};
\node (v13) at (0.3,2) {};
\node (v15) at (0.3,0) {};
\node (v16) at (0.3,-0.6) {};
\draw [thick,->] (v1) edge [out=0,in=180] (v9);
\draw [thick,->] (v1) edge [out=0,in=180] (v10);
\draw [thick,->]  (v1) edge (v2);
\draw [thick,->] (v3) edge [out=180,in=0] (v11);
\draw [thick,->] (v3) edge [out=180,in=0] (v12);
\draw [thick,->] (v3) edge (v4);
\draw [thick,->] (v13) edge (v14);
\draw [thick,->] (v15) edge (v16);
\end{tikzpicture}
\]
where the rewritten network is obtained by taking a
double pushout \cite{Cic}. By compactness, once we have
this $2$-morphism, we can swap the
roles of the inputs and outputs
to obtain the following $2$-morphism:
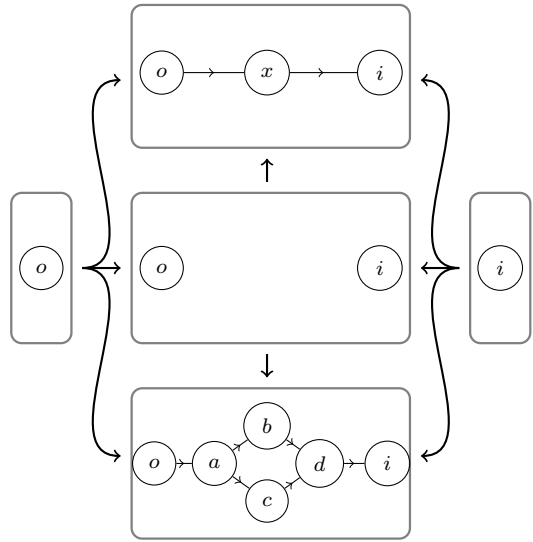
\[ 
\begin{tikzpicture}
%
%
\begin{scope}[shift={(-0.3,0)}]
\draw[rounded corners,line width=0.3mm, gray] (-2.8,0) rectangle (-2,2);
\node [style=circle,draw=black] at (-2.4,1) {\scriptsize $o$};
\end{scope}
%
%
\begin{scope}[shift={(0.5,-3)}]
\draw[rounded corners,line width=0.3mm, gray] (-2,3) rectangle (1.7,5);
\node [style=circle,draw=black] (3) at (1.3,4) {\scriptsize $i$};
\node [style=circle,draw=black] (4) at (-1.6,4) {\scriptsize $o$};
\end{scope}
%
%
\begin{scope}[shift={(0.8,-0.4)}]
\draw[rounded corners,line width=0.3mm, gray] (-2.3,3) rectangle (1.4,4.9);
\node [style=circle,draw=black] (1) at (-0.5,4) {\scriptsize $x$};
\node [style=circle,draw=black] (3) at (1,4) {\scriptsize $i$};
\node [style=circle,draw=black] (4) at (-1.9,4) {\scriptsize $o$};
\draw [->-] (1) to (3);
\draw [->-] (4) to (1);
\end{scope}
%
%
\begin{scope}[shift={(0.9,-5.6)}]
\draw[rounded corners,line width=0.3mm, gray] (-2.4,3) rectangle (1.3,5);
\node [style=circle,draw=black] (4) at (-2.1,4) {\scriptsize $o$};
\node [style=circle,draw=black] (3) at (1,4) {\scriptsize $i$};
\node [style=circle,draw=black] (v5) at (-1.3,4) {\scriptsize $a$};
\node [style=circle,draw=black] (v6) at (-0.6,4.5) {\scriptsize $b$};
\node [style=circle,draw=black] (v7) at (-0.6,3.5) {\scriptsize $c$};
\node [style=circle,draw=black] (v8) at (0.1,4) {\scriptsize $d$};
\draw [->-]  (4) to (v5);
\draw [->-] (v5) to (v6);
\draw [->-] (v5) to (v7);
\draw [->-] (v6) to (v8);
\draw [->-] (v7) to (v8);
\draw [->-] (v8) to (3);
\end{scope}
%
%
\begin{scope}[shift={(1,0)}]
\draw[rounded corners,line width=0.3mm, gray] (2,0) rectangle (2.8,2);
\node [style=circle,draw=black] at (2.4,1) {\scriptsize $i$};
\end{scope}
%
%
\node (v1) at (-2.3,1) {};
\node (v2) at (-1.5,1) {};
\node (v3) at (3,1) {};
\node (v4) at (2.2,1) {};
\node (v9) at (-1.5,3.5) {};
\node (v10) at (-1.5,-1.5) {};
\node (v11) at (2.2,3.5) {};
\node (v12) at (2.2,-1.5) {};
\node (v14) at (0.3,2.6) {};
\node (v13) at (0.3,2) {};
\node (v15) at (0.3,0) {};
\node (v16) at (0.3,-0.6) {};
\draw [thick,->] (v1) edge [out=0,in=180] (v9);
\draw [thick,->] (v1) edge [out=0,in=180] (v10);
\draw [thick,->]  (v1) edge (v2);
\draw [thick,->] (v3) edge [out=180,in=0] (v11);
\draw [thick,->] (v3) edge [out=180,in=0] (v12);
\draw [thick,->] (v3) edge (v4);
\draw [thick,->] (v13) edge (v14);
\draw [thick,->] (v15) edge (v16);
\end{tikzpicture}
\]
Hence there is no
substantial difference
between inputs and outputs.
Also, we can flip this diagram vertically to 
obtain the following $2$-morphism:
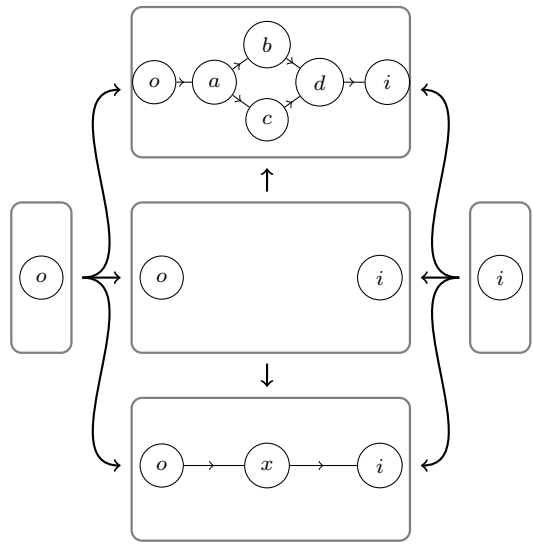
\[
\begin{tikzpicture}
%
%
\begin{scope}[shift={(-0.3,0)}]
\draw[rounded corners,line width=0.3mm, gray] (-2.8,0) rectangle (-2,2);
\node [style=circle,draw=black] at (-2.4,1) {\scriptsize $o$};
\end{scope}
%
%
\begin{scope}[shift={(0.5,-3)}]
\draw[rounded corners,line width=0.3mm, gray] (-2,3) rectangle (1.7,5);
\node [style=circle,draw=black] (3) at (1.3,4) {\scriptsize $i$};
\node [style=circle,draw=black] (4) at (-1.6,4) {\scriptsize $o$};
\end{scope}
%
%
\begin{scope}[shift={(0.8,-5.5)}]
\draw[rounded corners,line width=0.3mm, gray] (-2.3,3) rectangle (1.4,4.9);
\node [style=circle,draw=black] (1) at (-0.5,4) {\scriptsize $x$};
\node [style=circle,draw=black] (3) at (1,4) {\scriptsize $i$};
\node [style=circle,draw=black] (4) at (-1.9,4) {\scriptsize $o$};
\draw [->-] (1) to (3);
\draw [->-] (4) to (1);
\end{scope}

%
%
\begin{scope}[shift={(0.9,-0.4)}]
\draw[rounded corners,line width=0.3mm, gray] (-2.4,3) rectangle (1.3,5);
\node [style=circle,draw=black] (4) at (-2.1,4) {\scriptsize $o$};
\node [style=circle,draw=black] (3) at (1,4) {\scriptsize $i$};
\node [style=circle,draw=black] (v5) at (-1.3,4) {\scriptsize $a$};
\node [style=circle,draw=black] (v6) at (-0.6,4.5) {\scriptsize $b$};
\node [style=circle,draw=black] (v7) at (-0.6,3.5) {\scriptsize $c$};
\node [style=circle,draw=black] (v8) at (0.1,4) {\scriptsize $d$};
\draw [->-]  (4) to (v5);
\draw [->-] (v5) to (v6);
\draw [->-] (v5) to (v7);
\draw [->-] (v6) to (v8);
\draw [->-] (v7) to (v8);
\draw [->-] (v8) to (3);
\end{scope}
%
%
\begin{scope}[shift={(1,0)}]
\draw[rounded corners,line width=0.3mm, gray] (2,0) rectangle (2.8,2);
\node [style=circle,draw=black] at (2.4,1) {\scriptsize $i$};
\end{scope}
%
%
\node (v1) at (-2.3,1) {};
\node (v2) at (-1.5,1) {};
\node (v3) at (3,1) {};
\node (v4) at (2.2,1) {};
\node (v9) at (-1.5,3.5) {};
\node (v10) at (-1.5,-1.5) {};
\node (v11) at (2.2,3.5) {};
\node (v12) at (2.2,-1.5) {};
\node (v14) at (0.3,2.6) {};
\node (v13) at (0.3,2) {};
\node (v15) at (0.3,0) {};
\node (v16) at (0.3,-0.6) {};
\draw [thick,->] (v1) edge [out=0,in=180] (v9);
\draw [thick,->] (v1) edge [out=0,in=180] (v10);
\draw [thick,->]  (v1) edge (v2);
\draw [thick,->] (v3) edge [out=180,in=0] (v11);
\draw [thick,->] (v3) edge [out=180,in=0] (v12);
\draw [thick,->] (v3) edge (v4);
\draw [thick,->] (v13) edge (v14);
\draw [thick,->] (v15) edge (v16);
\end{tikzpicture}
\]
Thus the rewrite rules
are also symmetric.  

\section{Conclusion} 
\label{sec:Conclusion}

We have taken a closer look at 
spans of cospans, 
extending the results of
the first author \cite{Cic}
by finding a symmetric monoidal
and compact closed structure
on the bicategory $\bimonspcsp{T}$.
This structure is relevant
when using $\bimonspcsp{T}$ as 
a framework in which to study
various networks.  Generally speaking,
the symmetric monoidal structure
allows us to consider disjoint networks
as a single network by taking a coproduct of the two networks,
and the compact closed structure allows us 
to turn open networks around, 
indicating that there is no
substantial difference between
inputs and outputs
other than a shift in perspective.  
The primary advantage of 
having this structure 
to study networks is that
in cases where we are able to find a sub-bicategory
of $\bimonspcsp{T}$
which gives the networks that are being considered
a sense of
inputs and outputs,
we can freely generate compact closed
bicategories with various families
of $1$-morphisms and $2$-morphisms.
This was illustrated
in Section $\ref{sec:Applications}$
by the compact closed sub-bicategory
$\cat{Rewrite}$ of $\mathbf{MonicSp(Csp(Graph))}$
which provided inputs and outputs 
to graphs.

\section{Acknowledgments}
The authors would like to thank John Baez for his numerous careful readings and suggestions. Without his guidance and wisdom, this paper would not have been possible. Also for his patience and humor, which don't get the credit that they deserve. Lastly, the authors would like to thank Susan Niefield for serving as transmitting editor and an anonymous referee for helpful comments.

%
%


%
%

\end{document}